\documentclass[11pt,epsfig,amsfonts]{amsart}

\newtheorem{lm}{Lemma}[section]
\newtheorem{thm}{Theorem}[section]

\newtheorem{rmk}{Remark}
\numberwithin{equation}{section}
\makeatletter
\@namedef{subjclassname@2020}{2020 Mathematics Subject Classification}
\makeatother

\usepackage{color}
\usepackage{epsfig}



\author{Shuang Chen}
\address[Shuang Chen]
{School of Mathematics and Statistics \& Center for Mathematical Sciences  \\
Huazhong University of Sciences and Technology\\
Wuhan, Hubei 430074, P. R. China}
\email[S.~Chen]{schen@hust.edu.cn, matschen@163.com}

\author{Jun Shen}
\address[Jun Shen]
{School of Mathematics \\
Sichuan University\\
Chengdu, Sichuan 610064, P. R. China}
\email[J.~Shen]{junshen85@163.com}

\begin{document}

\begin{abstract}
In this paper, we study the dynamical  behaviors of neutral differential equations with small delays.
We first establish the existence and smoothness of the global inertial manifolds
for these equations. Then we further prove the smoothness of inertial manifolds with respect to small delays
for a certain class of neutral differential equations.
The method can be also applied to deal with more general small-delay systems.
Finally, we apply our main results to the van der Pol oscillator model
with small delay.

\end{abstract}
\title[Smooth inertial manifolds for NDE with small delays]{Smooth inertial manifolds for neutral differential equations with small delays
}

\maketitle

\baselineskip=12pt
\section{Introduction}
\label{sec-intr}
The evolution of an actual system is likely to depend
not only on the present state but also on the historical state.
Motivated by it, differential equations with delays have been widely studied.
Compared with differential equations without delays,
an interesting problem is to study the effects of small delays on the dynamical behaviors of systems.
Much effort has been made in the past few decades,
for example,
retarded differential equations were  studied  in \cite{Arino-Pituk,Casal-Corsi-Llave-20,Chicone03,Driver1968,Driver1976,Guoetal,Ouifki,Ryabov1965},
delay partial differential equations were considered in \cite{Bessaihetal,Faria-Huang,Li-Shi,Wangetal},
specific applications were  discussed  in \cite{Campbelletal,Erneuxetal,Fowler,Hill-Shafer-18,Li-Kloeden,Mao} and so on.
Recently,
several works paid more attentions to neutral differential equations with small delays,
see \cite{Chen-Shen-20,Gyori-Pitukl,Hale-Lunel2001,Hale-Lunel2002,Liu} and the references therein.

As shown in \cite{JKHale-Verduyn},
neutral differential equations can be seen as infinite-dimensional systems
if the dynamical behavior is described in the space of continuous functions.
In the  framework of infinite-dimensional systems,
one of interesting problems is the finite-dimensional reduction.
On this topic, it is of importance to study the existence and smoothness of inertial manifold.
As defined in \cite{Foiaseatal},
the inertial manifold is
a finite-dimensional smooth submanifold of the phase space,
which is invariant with respect to the family of solution operators,
possesses the exponential tracking property,
i.e., the inertial manifold attracts any trajectory that starts outside of the manifold exponentially fast.

In this paper, we will consider the existence and smoothness of the global inertial manifolds for neutral differential equations.
For some constant $r>0$,
let $\mathcal{C}$ denote the set of all continuous maps from $[-r,0]$ into $\mathbb{R}^n$,
which is a Banach space endowed with the supremum norm $|\phi|:=\sup_{\theta\in[-r,0]}|\phi(\theta)|$.
The notation $|\cdot|$ is always used to denote norms in different spaces, but no confusion should arise.
We first consider a nonautonomous neutral differential equation of the form
\begin{eqnarray}\label{NA-NDE-1}
\frac{d}{d t}\left\{x(t)-L(t)x_{t}\right\}= F(t,x_t),
\end{eqnarray}
where the section $x_t(\theta):= x(t+\theta)$ for $\theta\in[-r,0]$,
$F:\mathbb{R}\times \mathcal{C} \rightarrow  \mathbb{R}^n$ is a continuous map and there is a constant $K>0$ such that
\begin{eqnarray*}
|F(t,\phi)-F(t,\psi)|\leq K|\phi-\psi| \ \  \mbox{ for any }t\in \mathbb{R} \ \mbox{ and }  \ \phi, \psi\in \mathcal{C},
\end{eqnarray*}
and for each $t\in \mathbb{R}$, the operator $L(t):\mathcal{C}\to \mathbb{R}^n$
is defined by
\begin{eqnarray}\label{L-expres}
L(t)\phi=\int^{0}_{-r} d\eta(t,\theta)\phi(\theta) ,\ \ \
          \phi \in \mathcal{C}.
\end{eqnarray}
As was done in \cite[p.255]{JKHale-Verduyn}, we assume that the kernel
$\eta: \mathbb{R}\times \mathbb{R}\to \mathbb{R}^{n\times n}$ is measurable and normalized so that $\eta(t,\theta)$ satisfies
$\eta(t,\theta)=\eta(t,-r)$ for $\theta\leq -r$, $\eta(t,\theta)=0$ for $\theta\geq 0$,
$\eta(t,\cdot)$ is continuous from the left on $(-r,0)$ and has bounded variation uniformly in $t$
%
such that $t\mapsto L(t)\phi$ is continuous for each $\phi\in \mathcal{C}$.
Furthermore, the kernel $\eta$ is uniformly nonatomic at zero, i.e.,
for every $\epsilon>0$. there exists a constant $\delta>0$ such that
the total variation of $\eta(t,\cdot)$ on $[-\delta,0]$ is less than $\epsilon$ for all $t\in \mathbb{R}$.
It then follows from  \cite[Theorem 8.1, p.61]{JKHale-Verduyn} and \cite[Theorem 8.3, p.65]{JKHale-Verduyn}
that the Cauchy problem of equation (\ref{NA-NDE-1}) is well-posed.
Then for each $t_0\in \mathbb{R}$ and $\phi\in\mathcal{C}$,
there exists a unique solution $x(\cdot\,;t_0,\phi): [-r,+\infty) \rightarrow \mathbb{R}$
of equation (\ref{NA-NDE-1}) with the initial value $\phi$ at $t_0$,
i.e.,  $x(\cdot\,;t_0,\phi)$ with $x_{t_0}=\phi$ is a continuous map,
$x(t)-L(t)x_{t}$ is continuously differentiable and satisfies equation (\ref{NA-NDE-1}) on $[t_0,+\infty)$.

If $L(t)\phi=(0,...,0)^T\in \mathbb{R}^{n}$ for any $t\in \mathbb{R}$ and $\phi\in \mathcal{C}$,
then equation (\ref{NA-NDE-1}) is reduced to a retarded differential equation.
Ryabov (\cite{Ryabov1965}) and Driver (\cite{Driver1968}) obtained
that retarded differential equations with small delays has a Lipschitz inertial manifold.
Later on, Chicone (\cite{Chicone03}) generalized this result to $C^{1,1}$ inertial manifold
if $F$ is $C^{1,1}$ with respect to the second variable by using the Fiber Contraction Theorem (see \cite{Hirsch-Pugh}).
However, to approximate inertial manifolds,
we need obtain the higher-order smoothness of inertial manifolds.
To this end, in current paper, we will prove that neutral differential equation (\ref{NA-NDE-1})
with small delay $r$ possesses a global $C^{k,1}$ inertial manifold
if $F$ is $C^{k,1}$ with respect to the second variable.

We further assume that $L(t)\phi= A\phi(-r)$ for any $\phi\in \mathcal{C}$
and $F(t,x_t)=f(x(t),x(t-r))$ for $t\in \mathbb{R}$ and $r>0$.
Then equation (\ref{NA-NDE-1}) is rewritten as
\begin{eqnarray}\label{A-NDE-1}
\frac{d}{d t}\left\{x(t)-Ax(t-r)\right\}= f(x(t),x(t-r)),
\end{eqnarray}
where $A$ is an $n\times n$ real matrix and $f$ is Lipschitz continuous function, i.e.,
\begin{eqnarray*}
|f(y_1,z_1)-f(y_2,z_2)|\leq K \max\{|y_1-y_2|, \ |z_1-z_2|\}
\ \  \mbox{ for any } (y_1,z_1) \mbox{ and } (y_2,z_2) \mbox{ in } \mathbb{R}^{n}\times \mathbb{R}^{n}.
\end{eqnarray*}
In the second part of this paper, we will show that
the inertial manifold of equation (\ref{A-NDE-1}) with small delay is $C^{k,1}$
with respect to the delay $r$ if $f$ is $C^{k,1}$.

When we consider the differentiability of inertial manifold with respect to the delay $r$,
we need look for an appropriate space such that inertial manifold is smooth in $r$.
As early as 1991,  on the space $W^{1,\infty}([-r,0],\mathbb{R}^n)$,
where all functions are absolutely continuous and their derivatives are essentially bounded,
Hale and Ladeira (\cite{Hale-Ladeira}) used the Uniform Contraction Principle to prove that the solutions are $C^{k-1}$ in $r$ provided that the map $f$ is $C^k$.
Then Hartung and Turi (\cite{Hartung-Turi}) studied the differentiability of solutions
with respect to parameters
for state-dependent delay equations.
Chicone (\cite{Chicone03}) used the Fiber Contraction Theorem on a certain weighted Sobolev-type space and
proved that the inertial manifolds
of retarded differential equations with small delays are $C^{1,1}$ with respect to delays.
Subsequently, Chicone (\cite{Chicone04}) further studied the smoothness of inertial manifold for the delay equation $\dot x(t)= f(x(t),x(t-r))$
by the slow manifold of the $N$th-order
ordinary differential equation obtained by replacing the right-hand side
of the delay equation by the $N$th-order Taylor polynomial
of the function $\tau\mapsto f(x(t), x(t-\tau))$ at $\tau=0$.

Unlike Chicone (\cite{Chicone03}), we
studied the smoothness of inertial manifold for equation (\ref{NA-NDE-1}) (resp. equation (\ref{A-NDE-1}))
on some $C^{k,1}$ space
endowed with weighted supremum norm. To prove the completeness of these spaces with the metrics induced by the corresponding norms,
we use the Henry's lemma (see \cite[Lemma 6.1.6]{Henry1981}),
which is widely used to study the smoothness of invariant manifolds,
see Chow and Lu \cite{Chow-Lu},
Barreira and Valls \cite{Barreira-Valls2006,Barreira-Valls2007} and Elbialy\cite{Elbialy2001}.
Moreover, to guarantee that the Contraction Mapping Principle is valid,
the multivariate Fa\`a di Bruno Formula
(see \cite[Theorem 2.1]{Constaintine-Savits})
is applied to estimate the desired derivatives.

The paper is organized as follows.
In section \ref{sec-main result}, we will state the main results in this paper, see
Theorems \ref{thm-1}-\ref{thm-2}. Before proving these results, we first make some preparations in section 3.
In sections \ref{sec-pf-thm-1}-\ref{sec-pf-thm3},
we prove Theorems \ref{thm-1}-\ref{thm-2},
respectively.
To illustrate the applications of the main results,
we study the dynamics of the van der Pol oscillator model with small delay in section 7.
In section 8, we give a discussion for further study.

\section{main results}
\label{sec-main result}
In this section, we will state the main results in this paper.
We first consider a nonautonomous neutral differential equation of the form
\begin{eqnarray}\label{NA-NDE}
\frac{d}{d t}\left\{x(t)-L(t)x_{t}\right\}= F(t,x_t),
\end{eqnarray}
where the section $x_t(\theta):= x(t+\theta)$ for $\theta\in[-r,0]$.
To study the existence and the higher-order smoothness of the inertial manifold for equation (2.1) with small delay,
we need the following hypotheses.
\begin{enumerate}
\item[{\bf (H1)}] Let $L(t)$ be given by (\ref{L-expres}) and satisfy the conditions stated in the introduction.
Assume that $\sup_{t\in \mathbb{R}}|L(t)|=M$ with $0<M<1$
and $F:\mathbb{R}\times \mathcal{C} \to \mathbb{R}^n$ is a continuous map
and $C^{k}$ ($k\in \mathbb{N}$) in the second variable.
There exist positive constants $M_j$, $j=0,1,2,...,k+1$, such that
for any $t\in \mathbb{R}$ and $\phi,\psi \in \mathcal{C}$,
$|F(t,0)|\leq M_0 e^{\lambda|t|}$, $|D_2^jF(t,\phi)|\leq M_j$ for $j=1,...,k$
and $|D_2^kF(t,\phi)-D_2^k F(t,\psi)|\leq M_{k+1}|\phi-\psi|$,
where $D_2^{j}$ denotes the {\it j-}th partial derivative with respect to the second variable.
Let the constants $r_0$ and $x^{*}$ satisfy $0<M_1r_0<H(-\ln M/(k+1))$
and $x^{*}\in \left(x_1(r_0),-\ln M/(k+1)\right) \subset  \left(x_1(r_0),x_2(r_0)\right)$,
where $H$, $x_1(r_0)$ and $x_2(r_0)$ are defined in Appendix A.
The delay $r$ and the constant $\lambda$ satisfy $0<r\leq r_0$ and $r\lambda=x^*$.
\end{enumerate}

Since $x^{*}\in (x_1(r_0),x_2(r_0))$, by Appendix A, we find
$x^{*} e^{-x^{*}}-Mx^{*}>M_1r_0$, which implies $(Mx^{*}e^{x^{*}}+M_1r_0e^{x^{*}})/x^{*}<1$.
Then from the condition $r\lambda=x^{*}$ it follows
\begin{equation}\label{x-star}
\begin{split}
Me^{r\lambda}+M_1e^{r\lambda}/\lambda =&\,(Mr\lambda e^{r\lambda}+M_1re^{r\lambda})/(r\lambda)
=(Mx^{*} e^{x^{*}}+M_1re^{x^{*}})/x^{*} \\
\leq &\, (Mx^{*}e^{x^{*}}+M_1r_0e^{x^{*}})/x^{*}<1.
\end{split}
\end{equation}
Moreover, $x^{*}< -\ln M/(k+1)$ and $r\lambda=x^*$ imply that $Me^{(k+1)r\lambda}=Me^{(k+1)x^*}<1$.
\vskip 0.2cm
Our main results on equation (\ref{NA-NDE}) are summarized as following.
\begin{thm}\label{thm-1}
Assume that {\bf (H1)} holds.
Then there exists a constant $\delta$ with $0<\delta\leq r_0$ such that for equation (\ref{NA-NDE})
with $0<r<\delta$, there is a continuous map $\Psi: \mathbb{R}\times \mathbb{R}^{n}\to \mathbb{R}^n$ satisfying that
for each $\xi\in\mathbb{R}^{n}$,
$\Psi(\cdot\,,\xi)$ is the solution of equation (\ref{NA-NDE}) with the condition $x(0)-L(0)x_{0}=\xi$ and
for each $t\in\mathbb{R}$, $\Psi(t,\cdot\,)$ is a $C^{k,1}$ map on $\mathbb{R}^{n}$.
\end{thm}

\begin{thm}\label{thm-1-2}
Assume that {\bf (H1)} holds.
Then for the solution $x(\cdot\,; 0,\phi)$ of equation (\ref{NA-NDE}) with $x_{0}=\phi$,
there exists a unique $\xi\in \mathbb{R}^n$ such that
\begin{eqnarray*}
\sup_{t\geq0}|x(t; 0,\phi)-\Psi(t,\xi)|e^{\lambda t}<+\infty.
\end{eqnarray*}
\end{thm}

\begin{rmk}
Assume that {\bf (H1)} holds.
By Theorem \ref{thm-1} and Theorem \ref{thm-1-2},
we clearly see that the graph of $\Psi$ defined in Theorem \ref{thm-1}
forms the $C^{k,1}$ inertial manifold of equation (\ref{NA-NDE}) with small delay.
\end{rmk}

We further assume that $L(t)\phi= A\phi(-r)$ for any $\phi\in \mathcal{C}$
and $F(t,x_t)=f(x(t),x(t-r))$ for $t\in \mathbb{R}$ and $r>0$.
Then equation (\ref{NA-NDE}) is rewritten as
\begin{eqnarray}\label{A-NDE}
\frac{d}{d t}\left\{x(t)-Ax(t-r)\right\}= f(x(t),x(t-r)),
\end{eqnarray}
where $A$ is an $n\times n$ real matrix.
We will show that the inertial
manifold of equation (\ref{A-NDE}) is smooth with respect to the delay $r$. To this end, we need the following
hypothesis.
\begin{enumerate}
\item[{\bf (H2)}]
The function $f: \mathbb{R}^n\times \mathbb{R}^n \to \mathbb{R}^n$ is $C^k$
and satisfies that there exist constants $M_j$, $j=1,...,k+1$, such that
$|D^jf|\leq M_j$ and $|D^{k}f(y_1,z_1)-D^{k}f(y_2,z_2)|\leq M_{k+1}|(y_1,z_1)-(y_2,z_2)|$,
where $D^{j}$ denote the {\it j-}th derivative of $f$, and the matrix $A$ satisfies $|A|=M$ with $0<M<1/(2\cdot 3^{k})$.
Let the constants $r_0$ and $x^{*}$ satisfy $0<M_1r_0<H(-\ln(2\cdot 3^{k}M)/(k+1))$
and $x^{*}\in \left(x_1(r_0),-\ln(2\cdot 3^{k}M)/(k+1)\right) \subset \left(x_1(r_0), x_2(r_0) \right)$,
where $H$, $x_1(r_0)$ and $x_2(r_0)$ are defined in Appendix A.
The constant $\lambda$ satisfies $\delta\lambda=x^{*}$ for $0<\delta\leq r_0$.
\end{enumerate}

Without loss of generality, we assume that positive integer $k\geq 2$ in {\bf (H2)}.
Note that $0<M<1/(2\cdot 3^{k})<1$. This implies that $-\ln(2\cdot 3^{k}M)/(k+1)<-\ln M$.
By Appendix A, we obtain
\begin{eqnarray}\label{x-star-1}
(Mx^{*}e^{x^{*}}+M_1\delta e^{x^{*}})/x^{*}\leq(Mx^{*}e^{x^{*}}+M_1r_0e^{x^{*}})/x^{*}<1.
\end{eqnarray}
In addition, Note that $x^{*}<-\ln(2\cdot 3^{k}M)/(k+1)$ in {\bf (H2)}. Then
\begin{eqnarray}\label{x-star-2}
2\cdot 3^{k}Me^{(k+1)x^{*}}<1.
\end{eqnarray}

\vskip 0.2cm
In the end, we summarize our main results on equation (\ref{A-NDE}) as following.
\begin{thm}\label{thm-2}
Assume that {\bf (H2)} holds. Then for a sufficiently small $\delta$ with $0<\delta\leq r_0$,
there exists a continuous map $\Psi: \mathbb{R}\times (0\,,\delta) \times \mathbb{R}^{n}\to \mathbb{R}^n$
such that
\begin{itemize}
\item[(i)] for each $(r,\xi)\in(0\,,\delta) \times \mathbb{R}^{n}$,
$\Psi(\cdot\,,r,\xi)$ is the solution of equation (\ref{A-NDE}) with the condition $x(0)-Ax(-r)=\xi$;
\item[(ii)]for each $(t,r)\in\mathbb{R}\times (0\,,\delta)$,
$\Psi(t,r,\cdot\,)$ is a $C^{k,1}$ map on $\mathbb{R}^{n}$;
\item[(iii)] for each $\xi\in \mathbb{R}^{n}$, $\Psi(\cdot\,,\cdot\,, \xi)$ is a $C^{k,1}$ map on $\mathbb{R}\times (0\,,\delta)$.
\end{itemize}
\end{thm}
\begin{rmk}
\label{rmk-cut-off}
In many practical applications,
we are often interested in the dynamics of equation (\ref{NA-NDE}) (resp.\,(\ref{A-NDE}))
in some bounded regions, such as finding special bounded solutions including
periodic orbits, homoclinic loops and heteroclinic loops and so on.
Then we only require that the assumptions on the function $F$ (resp.\,$f$) in equation (\ref{NA-NDE}) (resp.\,(\ref{A-NDE}))
hold in some bounded regions.
In fact, by the cut-off technique we can modify the original equation (\ref{NA-NDE}) (resp.\,(\ref{A-NDE}))
such that the modified equation not only satisfies the assumptions, but also is consistent with the original one
in some large bounded regions.

On the other hand,
based on the higher-order smoothness obtained in Theorem \ref{thm-2},
we can also give an effective approximation of the inertial manifold for equation (\ref{A-NDE}).
In fact,
by the invariant property of the inertial manifold and the so-called post-Newtonian expansion
used for retarded differential equations with small delays in \cite{Chicone03,Chicone04},
we expand the restriction of equation  (\ref{A-NDE}) on the inertial manifold in a series with respect to $r$,
where a sequence of slow-fast systems is involved.
Then we could give an approximation of the inertial manifold for equation  (\ref{A-NDE}) with small delay
by analyzing the slow manifolds of these slow-fast systems.
\end{rmk}

\section{Preliminaries}
Before proving our main results, we first make some preliminaries in this section.
Let $B_{i},\,i=1,..., n$ and $E$ be Banach spaces and $B=B_1\times\cdot\cdot\cdot\times B_n$ be product space with the norm
$|v|:=\max_{1\leq i\leq n}|v_i|$ for ${v}=(v_1,...,v_n)\in B$.
Let $U$ denote the open subset of $B$ and $h$ be a $C^{k}$ map from $U$ to $E$.
As shown in \cite[p.181]{Dieudonne},
the {\it k-}th derivative $D^k h$ and the partial derivative with respect to the {\it j-}th variable $D_{j}h$
are identified to one of elements in $L^{k}(B; E)$ and $L(B_j; E)$, respectively,
where $L^{k}(B; E)$ denotes the set of all $k$  multilinear continuous maps from $B$ to $E$,
which is a Banach space equipped with the norm
$|\mathcal{L}|=\sup\left\{K\in \mathbb{R}: |\mathcal{L}(u_1,...,u_k)|\leq K |u_1|\cdot\cdot\cdot|u_k|\mbox{ for any } u_i\in B\right\}$
for each  $\mathcal{L}\in L^{k}(B;E)$,
and $L(B_j; E)$ denotes the set of all continuous linear maps from $B_{j}$ to $E$.
Let $\boldsymbol{\nu}=(\nu_1,...,\nu_n)\in \mathbb{N}_0^n$
with $|\boldsymbol{\nu}|=\nu_1+\cdot\cdot\cdot+\nu_n\leq k$ and
$\mathbb{N}_0$ denote the set of nonnegative integers.
$D^{\boldsymbol{\nu}}_{u}h:=\frac{\partial^{|\boldsymbol{\nu}|}}{\partial u_1^{\nu_1}...\partial u_n^{\nu_n}}h$
is said to be the partial derivative of order $|\boldsymbol{\nu}|$ of $h$
and assume $D^{\boldsymbol{\nu}}_{u}h=h$ for $\boldsymbol{\nu}=(0,...,0)$.
The following lemma can be found in \cite[p.182]{Dieudonne}.
\begin{lm}\label{derivt}
Let $U$ be an open set of $\mathbb{R}^n$ and $E$ denote a Banach space.
If the map $h: U \rightarrow E$ is $C^k$,
then for each $x\in U$ and $\xi_i=(\xi_{i1},...,\xi_{in})^{T}\in \mathbb{R}^n$, $i=1,...,k$, we have
$$
D^kh(x)(\xi_1,...,\xi_k)=\sum_{(j_1,...,j_k)} D_{j_1}D_{j_2}\cdot\cdot\cdot D_{j_k}h(x)\xi_{1,j_1}\cdot\cdot\cdot \xi_{k,j_k},
$$
where the sum is extended to all $n^k$ distinct sequences $(j_1,...,j_k)$ of integers
from the set of indices $\{1,...,n\}$.
\end{lm}

As was done in \cite{Elbialy2000},
for each $0<\alpha\leq1$ and $k\in \mathbb{N}$,
where $\mathbb{N}$ denotes the set of positive integers,
let $C^{k,\alpha}(U,E)$ be the set of $C^k$ maps $\phi: U\to E$ satisfying
\begin{eqnarray*}
H_{\alpha}(D^k\phi):=\sup_{x\neq x',\, x,\, x'\in U}\frac{|D^k \phi(x)-D^k \phi(x')|}{|x-x'|^{\alpha}}<\infty.
\end{eqnarray*}
For any $b>0$, we define the Banach space
$$C^{k,\alpha}_{b}:=\{\phi\in C^{k,\alpha}_{b}(U,E): |\phi|_{k,\alpha}\leq b\},$$
where
$|\phi|_{k,\alpha}:=\max\left\{|\phi|_{\infty},|D\phi|_{\infty},..., |D^k\phi|_{\infty}, H_{\alpha}(D^k\phi)\right\}$
and $|\cdot|_{\infty}$ denotes the supremum norm.

The following lemma is from \cite{Elbialy2000}.
\begin{lm}\label{Henry-lm} {\rm({\bf Henry's lemma})}
{\rm (i)}
Suppose that  a sequence $\{\phi_m\}_{m=1}^{+\infty} \subset C^{k,\alpha}_{b}$
and a map $\phi: U \to E$ satisfy $|\phi_m-\phi|_{\infty}\to 0$ as $m\to +\infty$,
then $\phi\in C^{k,\alpha}_{b}$ and $D^k\phi_m(u)\to D^k\phi(u)$ as $m\to+\infty$ for each $u\in U$.
{\rm (ii)}
Let $U_0\subset U$ be a subset which is uniformly bounded away from the boundary of $U$.
Then $D^k\phi_m(u)\to D^k\phi(u)$ uniformly on $U_0$.
\end{lm}

Assume that $B_1=\cdots =B_n=E=\mathbb{R}$. To state a multivariate Fa\`a di Bruno Formula (see \cite{Constaintine-Savits}),
we need introduce some notations.
Let $\boldsymbol{\nu}!=\prod_{i=1}^{n}(\nu_i!)$ and $x^{\boldsymbol{\nu}}=\prod_{i=1}^{n}x_i^{\nu_i}$
for $\boldsymbol{\nu}=(\nu_1,...,\nu_n)\in\mathbb{N}^n_0$ and $x=(x_1,...,x_n)\in \mathbb{R}^n$.
Set $\boldsymbol{\nu}=(\nu_1,...,\nu_n)$ and $\boldsymbol{\mu}=(\mu_1,...,\mu_n)$ belong to $\mathbb{N}^n_0$.
We write $\boldsymbol{\mu} \prec \boldsymbol{\nu}$ if one of the following holds:
{(i)} $|\boldsymbol{\mu}|<|\boldsymbol{\nu}|$;
{(ii)} $|\boldsymbol{\mu}|=|\boldsymbol{\nu}|$ and $\mu_1<\nu_1$;
{(iii)} $|\boldsymbol{\mu}|=|\boldsymbol{\nu}|$, $\mu_1=\nu_1$,..., $\mu_j=\nu_j$ and $\mu_{j+1}<\nu_{j+1}$ for some $1\leq j<n$.

Let
$h(x_1,...,x_n)=g(g_1(x_1,...,x_n),...,g_m(x_1,...,x_n))$ for $x=(x_1,...,x_n)\in \mathbb{R}^n$,
where $g_i:\mathbb{R}^n\to \mathbb{R}$, $i=1,..,m$, are $C^k$ maps defined in a neighborhood of $x^0=(x^0_1,...,x^0_n)$
and $g:\mathbb{R}^m\to \mathbb{R}$ is a $C^k$ map defined in a neighborhood of $y^0:=(g_1(x^0),...,g_m(x^0))\in\mathbb{R}^m$.
We define
$h_{\boldsymbol{\nu}}:=D^{\boldsymbol{\nu}}_{x}h(x^0)$,
$g_{\boldsymbol{\omega}}:=D^{\boldsymbol{\omega}}_{y}g(y^0)$,
$g^{(i)}_{\boldsymbol{\mu}}:=D^{\boldsymbol{\mu}}_{x}g_{i}(x^0)$
and ${\bf g}_{\boldsymbol{\mu}}:=(g^{(1)}_{\boldsymbol{\mu}}, ...,g^{(m)}_{\boldsymbol{\mu}})$. Set $0^0=1$.
The following lemma shows the explicit expression of an arbitrary partial derivative of composite function $h$, which can be found
in \cite[Theorem 2.1]{Constaintine-Savits}.
\begin{lm}\label{partial-devt}
Let $h$ be given as above. Then we have
\begin{eqnarray*}
h_{\boldsymbol{\nu}} =
\!\!\sum_{1\leq|\boldsymbol{\omega}|\leq |\boldsymbol{\nu}|}g_{\boldsymbol{\omega}}\sum_{s=1}^{|\boldsymbol{\nu}|}\sum_{p_s(\boldsymbol{\nu},\boldsymbol{\omega})}(\boldsymbol{\nu}!)
\prod_{j=1}^{s}\frac{({\bf g}_{\boldsymbol{l_j}})^{\boldsymbol{k_j}}}{(\boldsymbol{k_j}!)(\boldsymbol{l_j}!)^{|\boldsymbol{k_j}|}},
\end{eqnarray*}
where
\begin{eqnarray*}
p_s(\boldsymbol{\nu},\boldsymbol{\omega})\!\!\!&=&\!\!\!
\{(\boldsymbol{k_1},...,\boldsymbol{k_s};\boldsymbol{l_1},...,\boldsymbol{l_s}):
|\boldsymbol{k_j}|>0, 0\prec \boldsymbol{l_1}\prec...\prec \boldsymbol{l_s},\\
\!\!\!& &\!\!\!
 \sum_{j=1}^{s} \boldsymbol{k_j}=\boldsymbol{\omega} \mbox{ and }
\sum_{j=1}^{s} |\boldsymbol{k_j}|\boldsymbol{l_j}=\boldsymbol{\nu}
\}.
\end{eqnarray*}
\end{lm}

\section{Proof of Theorem \ref{thm-1}}
\label{sec-pf-thm-1}

For any fixed constant $d>0$, set  $V^{0}_d:=\{\xi\in \mathbb{R}^n: |\xi|<d\}$.
Let $\mathcal{B}_{d,\lambda}$ denote
the set of continuous maps from $\mathbb{R}\times V^{0}_d$ to $\mathbb{R}^n$ and
each $x\in \mathcal{B}_{d,\lambda}$ is $C^k$ in the second variable
and satisfies that for some constants $\beta_{j}>0$,
\begin{eqnarray}
|x|_{\mathcal{B}_{d,\lambda}}\!\!\!&:=&\!\!\!
\sup_{(t,\xi)\in \mathbb{R}\times V^{0}_d} |x(t,\xi)|e^{-\lambda|t|} \leq \beta_0,\label{norm-x}\\
|x|_{\mathcal{B}_{d,\lambda},j}\!\!\!&:=&\!\!\!
\sup_{(t,\xi)\in \mathbb{R}\times V^{0}_d} |D_2^jx(t,\xi)|e^{-j\lambda|t|} \leq \beta_j, \ \  j=1,2,...,k,\label{bd-xj}\\
|x|_{\mathcal{B}_{d,\lambda},k+1}\!\!\!\!\!&:=&\!\!\!\!
\sup_{t\in \mathbb{R},\, \xi_1,\,\xi_2\in V^{0}_d,\,\xi_1\neq\xi_2}
\!\!\frac{|D_2^k x(t,\xi_1)-D_2^k x(t,\xi_2)|}{|\xi_1-\xi_2|}e^{-(k+1)\lambda|t|} \leq \beta_{k+1},\label{bd-lip-x}
\end{eqnarray}
where the constant $\lambda$ is given as in {\bf (H1)}.
\begin{lm}\label{lm-4-complete}
Let the map $\rho$ be defined in the form $\rho(x,y)=|x-y|_{\mathcal{B}_{d,\lambda}}$ for any $x, y\in \mathcal{B}_{d,\lambda}$.
Then the set $\mathcal{B}_{d,\lambda}$ endowed with the metric $\rho$ is a complete metric space.
\end{lm}
\begin{proof}
Clearly, the map $\rho$ on $\mathcal{B}_{d,\lambda} \times \mathcal{B}_{d,\lambda}$ is well defined and induces a metric.
In the following we prove the completeness of the metric space $(\mathcal{B}_{d,\lambda}, \rho)$.
Let $\{g_m\}_{m=1}^{+\infty}$ be a Cauchy sequence of $\mathcal{B}_{d,\lambda}$,
that is, for any $\epsilon>0$, there is a positive integer $N(\epsilon)$
such that for any positive integer $m, m'\geq N(\epsilon)$,
\begin{eqnarray*}\label{Cauchy-seq-1}
\rho(g_{m'},g_m)=|g_{m'}-g_m|_{\mathcal{B}_{d,\lambda}}=
\sup_{(t,\xi)\in \mathbb{R}\times V^{0}_d} |g_{m'}(t,\xi)-g_m(t,\xi)|e^{-\lambda|t|}<\epsilon.
\end{eqnarray*}
Then $\{g_m(t,\xi)e^{-\lambda|t|}\}_{m=1}^{+\infty}$ is a Cauchy sequence in
the Banach space $C_b(\mathbb{R}\times V^{0}_d,\mathbb{R}^n):=
\{f \in C(\mathbb{R}\times V^{0}_d,\mathbb{R}^n): \sup_{(t,\xi)\in \mathbb{R}\times V^{0}_d} |f(t,\xi)|<+\infty\}$.
Let $\tilde{g}_0(t, \xi)$  be the limit of $g_{m}(t,\xi)e^{-\lambda|t|}$
in $C_b(\mathbb{R}\times V^{0}_d,\mathbb{R}^n)$. Set $g_0(t,\xi):=\tilde{g}_0(t,\xi)e^{\lambda|t|}$. This implies that
$\rho(g_{m},g_{0})\to 0$ as $m\to +\infty$.

Next we prove that $g_0\in\mathcal{B}_{d,\lambda}$. For any $g\in \mathcal{B}_{d,\lambda}$ and fixed $t\in \mathbb{R}$,
we define the map $\widetilde{g}:V^{0}_d\to \mathbb{R}^n$ in the form
$\widetilde{g}(\xi)=g(t,\xi)$, then
$|\widetilde{g}(\xi)|=|g(t,\xi)|\leq |g|_{\mathcal{B}_{d,\lambda}}\,e^{\lambda|t|}\leq \beta_0e^{\lambda|t|}.$
Recall that $\rho(g_{m},g_{0})\to 0$ as $m\to +\infty$.
Then we have $|\widetilde{g}_m-\widetilde{g}_0|_{\infty}\to0$ as $m\to+\infty$.
From (\ref{norm-x})-(\ref{bd-lip-x}) it follows that
$\{\widetilde{g}_m\}_{m=1}^{+\infty}\subset C_{\beta(t)}^{k,1}$,
where $\beta(t)=\max\{\beta_{0}e^{\lambda|t|}, \beta_je^{j\lambda|t|} \mbox{ for } j=1,2,...,k+1\}$.
Hence, by Lemma \ref{Henry-lm} we obtain
$\widetilde{g}_0\in C_{\beta(t)}^{k,1}$ and $D^j\widetilde{g}_m(\xi)\to D^j\widetilde{g}_0(\xi)$ as $m\to +\infty$
for each $\xi\in V^{0}_d$ and $j=1,2,...,k$.
Together with (\ref{norm-x})-(\ref{bd-lip-x}) again, we obtain $g_0\in\mathcal{B}_{d,\lambda}$.
Therefore, the proof is complete.
\end{proof}

On the space $\mathcal{B}_{d,\lambda}$, we define a map $\mathcal{T}$ in the following form
\begin{eqnarray}\label{map-1}
\mathcal{T}(x)(t,\xi)=\xi+L(t)x_{t}+\int_0^{t} F(s,x_s)ds.
\end{eqnarray}
We will show that the unique fixed point of $\mathcal{T}$ in $\mathcal{B}_{d,\lambda}$ is the desired map $\Psi$ in Theorem \ref{thm-1}.
To this end, we first need prove the following lemmas.

\begin{lm}\label{lm-GB4-1}
Assume that {\bf (H1)} holds. Then
for any $x\in \mathcal{B}_{d,\lambda}$ and $(t,\xi) \in \mathbb{R}\times V^{0}_d$,
the following estimates hold:
\begin{itemize}
\item[(i)]
$|\mathcal{T}(x)(t,\xi)|e^{-\lambda|t|}\leq d+M_0/\lambda+(Me^{r\lambda}+M_1e^{r\lambda}/\lambda)\beta_0.$
\item[(ii)]
$|D_2\mathcal{T}(x)(t,\xi)|e^{-\lambda|t|}\leq 1+(Me^{r\lambda}+M_1e^{r\lambda}/\lambda)\beta_1.$
\item[(iii)]
For any positive integer $m$ with $2\leq m\leq k$, we have
$$|D_2^{m}\mathcal{T}(x)(t,\xi)|e^{-m\lambda|t|}\leq M\beta_{m}e^{mr\lambda}+A_m/\lambda,$$
where
\begin{eqnarray*}
A_m\!\!\!&=&\!\!\!(m-1)!e^{mr\lambda}\sum_{j=1}^{m} M_j\sum_{p(m,j)}\prod_{i=1}^{m}\frac{\beta_i^{\omega_i}}{(\omega_i!)(i!)^{\omega_i}},\\
p(m,j)\!\!\!&=&\!\!\!\left\{(\omega_1,...,\omega_m): \omega_i\in\mathbb{N}_0, \sum_{i=1}^{m}\omega_i=j, \sum_{i=1}^{m}i\omega_i=m\right\}.
\end{eqnarray*}
\end{itemize}
\end{lm}

\begin{proof}
For any $x\in \mathcal{B}_{d,\lambda}$ and $(t,\xi) \in \mathbb{R}\times V^{0}_d$, we have
\begin{eqnarray*}
|\mathcal{T}(x)(t,\xi)|\!\!\!&\leq&\!\!\! |\xi|+|L(t)x_{t}|+|\int_0^{t} |F(s,x_s)-F(s,0)|ds|+|\int_0^{t} |F(s,0)|ds|\\
                       \!\!\!&\leq&\!\!\! d+M\beta_0e^{\lambda(r+|t|)}+M_1|\int_0^{t} \beta_0e^{\lambda(r+|s|)}ds|+M_0e^{\lambda|t|}/\lambda\\
             \!\!\!&\leq&\!\!\! \left(d+M_0/\lambda+(Me^{r\lambda}+M_1e^{r\lambda}/\lambda)\beta_0\right)e^{\lambda|t|}.
\end{eqnarray*}
Thus, result (i) is proved.

By Leibniz's Rule (\cite[Theorem 8.11.2, p.177]{Dieudonne}) and the linearity of $L(t)$,
we see that
\begin{eqnarray*}
D_2\mathcal{T}(x)(t,\xi)=I+L(t)(D_2x_{t})+\int_0^{t} D_2F(s,x_s)D_2 x_s ds.
\end{eqnarray*}
Then we have
\begin{eqnarray*}
|D_2\mathcal{T}(x)(t,\xi)|\!\!\!&\leq&\!\!\! 1 +M|D_2x_{t}|+|\int_0^{t} |D_2F(s,x_s)||D_2 x_s| ds|\\
                          \!\!\!&\leq&\!\!\! 1 +M\beta_{1}e^{\lambda(r+|t|)}+|\int_0^{t} M_1\beta_1e^{\lambda(r+|s|)}ds|
                          \leq \left(1+(Me^{r\lambda}+M_1e^{r\lambda}/\lambda)\beta_1\right)e^{\lambda|t|}.
\end{eqnarray*}
Thus, result (ii) is proved.

For any positive integer $m$ with $2\leq m\leq k$,
by the Leibniz's Rule and the univariate Fa\`a di Bruno Formula
(see (1.1) in \cite[p.503]{Constaintine-Savits}), we have
\begin{eqnarray*}
|D^m_2\mathcal{T}(x)(t,\xi)|\!\!\!&\leq&\!\!\!|L(t)||D_2^{m}x_{t}|+m!|\int_0^{t} \sum_{j=1}^{m} |D^j_2F(s,x_s)|\sum_{p(m,j)}\prod_{i=1}^{m}
\frac{|D^{i}_2 x_s|^{\omega_i}}{(\omega_i!)(i!)^{\omega_i}} ds|\\
              \!\!\!&\leq&\!\!\!
              M\beta_{m}e^{m\lambda(r+|t|)}+m!|\int_0^{t}\sum_{j=1}^{m} M_j\sum_{p(m,j)}\prod_{i=1}^{m}\frac{(\beta_ie^{i\lambda(r+|s|)})^{\omega_i}}{(\omega_i!)(i!)^{\omega_i}}ds|\\
              \!\!\!&\leq&\!\!\!
              M\beta_{m}e^{m\lambda(r+|t|)}+m!e^{mr\lambda}\sum_{j=1}^{m} M_j\sum_{p(m,j)}\prod_{i=1}^{m}\frac{\beta_i^{\omega_i}}{(\omega_i!)(i!)^{\omega_i}}|\int_0^{t}  e^{m\lambda|s|} ds|\\
              \!\!\!&\leq&\!\!\!
              \left(M\beta_{m}e^{mr\lambda}+A_{m}/\lambda\right)e^{m\lambda|t|}.
\end{eqnarray*}
Thus, result (iii) is proved.
Then the proof is complete.
\end{proof}

\begin{lm}\label{lm-GB4-2}
For any $t\in \mathbb{R}$, $\xi_1,\,\xi_2\in \mathbb{R}^{n}$ and $x\in \mathcal{B}_{d,\lambda}$, we have
$$
|D_2^k\mathcal{T}(x)(t,\xi_1)-D_2^k\mathcal{T}(x)(t,\xi_2)|
\leq \left(M\beta_{k+1}e^{(k+1)r\lambda}+A_{k+1}/\lambda\right) e^{(k+1)\lambda|t|}|\xi_1-\xi_2|,
$$
where  $A_{k+1}:=(A^1_{k+1}+A^2_{k+1})/(k+1)$,
\begin{eqnarray*}
A^1_{k+1}\!\!\!&:=&\!\!\!
\beta_1e^{(k+1)r\lambda}(k!)\sum_{m=1}^{k} M_{m+1}\sum_{p(k,m)}\prod_{i=1}^{k}\frac{\beta_i^{\omega_i}}{(\omega_i!)(i!)^{\omega_i}},\\
A^2_{k+1}\!\!\!&:=&\!\!\!
e^{(k+1)r\lambda}(k!)\sum_{m=1}^{k} M_{m}\sum_{p(k,m)}\sum_{j=1}^{k}
\frac{\omega_{j}\beta_{j+1}\beta_j^{\omega_j-1}}{(\omega_j!)(j!)^{\omega_j}}
\left(\prod_{i=0}^{j-1}\frac{\beta_i^{\omega_i}}{(\omega_i!)(i!)^{\omega_i}}\right)
\left(\prod_{i=j+1}^{k+1}\frac{\beta_i^{\omega_i}}{(\omega_i!)(i!)^{\omega_i}}\right),
\end{eqnarray*}
and $\omega_{0}=\omega_{k+1}:=0$.
\end{lm}
\begin{proof}
For any $\xi_1,\,\xi_2\in \mathbb{R}^{n}$ and $x\in \mathcal{B}_{d,\lambda}$,
let $y_t:=x_t(\cdot\,,\xi_1)$, $z_t:=x_t(\cdot\,,\xi_2)$,
$G(t,\xi_1):=F(t,y_t)$ and $G(t,\xi_2):=F(t,z_t)$.
 By Appendix B, we have
\begin{eqnarray}\label{est-G-dev}
|D^k_2G(t,\xi_1)-D^k_2G(t,\xi_2)|\leq I_1+I_2,
\end{eqnarray}
where
\begin{equation}
\begin{split}
I_1:=&\,k!\sum_{m=1}^{k}|D^m_2F(t,y_t)-D^m_2F(t,z_t)|
     \sum_{p(k,m)}\prod_{i=1}^{k}\frac{|D^{i}_2y_t|^{\omega_i}}{(\omega_i!)(i!)^{\omega_i}},\\
I_2:=&\,k!\sum_{m=1}^{k} |D^m_2F(t,z_t)|\sum_{p(k,m)}\sum_{j=1}^{k} \frac{Q_{j}}{(\omega_j!)(j!)^{\omega_j}}
\left(\prod_{i=0}^{j-1} \frac{|D^{i}_2z_t|^{\omega_i}}{(\omega_i!)(i!)^{\omega_i}}\right)
\left(\prod_{i=j+1}^{k+1} \frac{|D^{i}_2y_t|^{\omega_i}}{(\omega_i!)(i!)^{\omega_i}}\right),\label{I2}\\
Q_{j}:=&\, |D^{j}_2y_t-D^{j}_2z_t|\sum_{l=0}^{\omega_{j}-1}|D^{j}_2z_t|^{l}|D^{j}_2y_t|^{\omega_{j}-l-1}.
\end{split}
\end{equation}
For $I_1$, we observe that
\begin{equation}
\begin{split}
I_1
\leq&\,k!\sum_{m=1}^{k}M_{m+1}|y_t-z_t|
\sum_{p(k,m)}\prod_{i=1}^{k}\frac{\beta_i^{\omega_i}e^{i\omega_i\lambda(r+|t|)}}{(\omega_i!)(i!)^{\omega_i}}\\
\leq&\,
k!\sum_{m=1}^{k}M_{m+1}\beta_1e^{\lambda(r+|t|)}|\xi_1-\xi_2|
\sum_{p(k,m)}\prod_{i=1}^{k}\frac{\beta_i^{\omega_i}e^{i\omega_i\lambda(r+|t|)}}{(\omega_i!)(i!)^{\omega_i}}=    A_{k+1}^1 e^{(k+1)\lambda|t|}|\xi_1-\xi_2|\label{I1-est}.
\end{split}
\end{equation}
To estimate $I_2$, by (\ref{bd-xj}) and (\ref{bd-lip-x}) we note that
\begin{eqnarray*}
Q_{j}
\!\!\!&\leq&\!\!\! \beta_{j+1}e^{(j+1)\lambda(r+|t|)}|\xi_1-\xi_2|
\sum_{l=0}^{\omega_{j}-1}(\beta_je^{j\lambda(r+|t|)})^{l}(\beta_je^{j\lambda(r+|t|)})^{\omega_{j}-l-1}\nonumber\\
\!\!\!&\leq&\!\!\! \omega_{j}\beta_{j+1}\beta_j^{\omega_{j}-1} e^{(1+j\omega_{j})\lambda(r+|t|)}|\xi_1-\xi_2|,
\end{eqnarray*}
together with  (\ref{bd-xj}), (\ref{I2}) and {\bf (H1)}, we have
\begin{equation}
\begin{split}
I_2
\leq&\,
(k!)\sum_{m=1}^{k}\! M_{m}\!\!\sum_{p(k,m)}\sum_{j=1}^{k}
\frac{\omega_{j}\beta_{j+1}\beta_j^{\omega_j-1}e^{(1+j\omega_{j})\lambda(r+|t|)}|\xi_1-\xi_2|}{(\omega_j!)(j!)^{\omega_j}}
\\
&\, \times
\left(\prod_{i=0}^{j-1}\!\frac{(\beta_ie^{i\lambda(r+|t|)})^{\omega_i}}{(\omega_i!)(i!)^{\omega_i}}\right)
\left(\prod_{i=j+1}^{k+1}\!\!\frac{(\beta_ie^{i\lambda(r+|t|)})^{\omega_i}}{(\omega_i!)(i!)^{\omega_i}}\right)
\leq A_{k+1}^2e^{(k+1)\lambda|t|}|\xi_1-\xi_2|.\label{I2-est}
\end{split}
\end{equation}
In the end, applying the Leibniz's reule, in view of (\ref{bd-lip-x}), (\ref{I1-est}) and (\ref{I2-est}) we have
\begin{eqnarray*}
\lefteqn{|D_2^k\mathcal{T}(x)(t,\xi_1)-D_2^k\mathcal{T}(x)(t,\xi_2)|}\\
\!\!\!&\leq&\!\!\!
|L(t)||D_2^{k}y_{t}-D_2^{k}z_{t}|+
|\int_0^{t} \left(D^k_2G(s,\xi_1)-D^k_2G(s,\xi_2)\right) ds|\\
\!\!\!&\leq&\!\!\!\!\!
M\beta_{k+1}e^{(k+1)\lambda(r+|t|)}|\xi_1-\xi_2|+\!\!\int_0^{t}\!\!(A_{k+1}^1+A_{k+1}^2) e^{(k+1)\lambda|s|}|\xi_1-\xi_2| ds|\\
\!\!\!&\leq&\!\!\!\!\!
\left(M\beta_{k+1}e^{(k+1)r\lambda}+A_{k+1}/\lambda\right) e^{(k+1)\lambda|t|}|\xi_1-\xi_2|.
\end{eqnarray*}
Then Lemma \ref{lm-GB4-2} is established.
\end{proof}
In the end of this section, we prove Theorem \ref{thm-1} by the Contraction Mapping Principle.
\begin{proof}[Proof of Theorem \ref{thm-1}]
The proof of this theorem is divided into four steps.
\vskip 0.2cm
\noindent{\bf Step (i).}
We first choose a suitable constant $\delta$ with $0<\delta\leq r_{0}$ to give the smallness condition.
Recall that the constants $r_0$ and $x^{*}$ defined as in {\bf (H1)}
only depend on $M$ and $M_{1}$. Let the constants $\beta_{j}$, $j=1,...,k+1$, satisfy the following conditions:
\begin{eqnarray*}
\beta_1\geq x^{*}/(x^{*}-Mx^{*}e^{x^{*}}-M_1r_0e^{x^{*}})>0,\ \ \ \beta_{j}>0, \ \ \  j=2,...,k+1.
\end{eqnarray*}
By (\ref{x-star}) we also have $\beta_{1}>0$.
We define the constant $\delta$ by
\begin{eqnarray*}
\delta:=\min\left\{r_0, x^{*}(1-Me^{mx^{*}})\beta_{m}/A_{m} \mbox{ for } m=2,...,k+1\right\},
\end{eqnarray*}
where $A_m$\,s are defined in Lemmas \ref{lm-GB4-1} and  \ref{lm-GB4-2}.
Note that $A_m$\,s are only determined by $\beta_{j}$ for $j=1,...,m$,
and $1-Me^{mx^{*}}>0$ for each $m=2,...,k+1$,
then $\delta$ is well defined and satisfies $\delta>0$.
\vskip 0.2cm
\noindent{\bf Step (ii).}
Secondly, for  the delay $r$ satisfying the smallness condition $r\in (0,\delta)$ and each fixed $d>0$,
we construct the desired complete metric space $\mathcal{B}_{d,\lambda}$,
where the constant $\lambda=x^{*}/r$.
By (\ref{x-star}) we can take a sufficiently large $\beta_{0}$ such that the following inequality holds:
\begin{eqnarray}\label{beta-0-restr}
d+M_0r_0/x^*+\beta_0(Mx^{*}e^{x^{*}}+M_1r_0e^{x^{*}})/x^{*}\leq \beta_0.
\end{eqnarray}
Define the constants $\beta_{j}$, $j=0,...,k+1$, associated with $\mathcal{B}_{d,\lambda}$ by the above way.
Then by Lemma \ref{lm-4-complete},
the set $\mathcal{B}_{d,\lambda}$ endowed with the metric $\rho$, which is induced by (\ref{norm-x}), is a complete metric space.
\vskip 0.2cm
\noindent{\bf Step (iii).}
Thirdly, we prove that the operator $\mathcal{T}$ defined by (\ref{map-1}) maps $\mathcal{B}_{d,\lambda}$ to itself.
Clearly, for each $x\in \mathcal{B}_{d,\lambda}$,
$\mathcal{T}(x)$ is a continuous map from $\mathbb{R}\times V^{0}_d$ to $\mathbb{R}^n$.
Note that  $F$ and $x$ are $C^k$ maps in the second variable,
and $L(t)$ is a linear map for each $t\in \mathbb{R}$.
Then  $\mathcal{T}(x)$ is a $C^k$ map with respect to the second variable.
For $r\in (0,\delta)$, by the condition $r\lambda=x^{*}$, (\ref{x-star}) and (\ref{beta-0-restr}) we have
\begin{eqnarray*}
d+M_0/\lambda+(Me^{r\lambda}+M_1e^{r\lambda}/\lambda)\beta_0
\leq d+M_0r_0/x^{*}+\beta_0(Mx^{*} e^{x^{*}}+M_1r_0e^{x^{*}})/x^{*}\leq \beta_0.
\end{eqnarray*}
In view of Lemma \ref{lm-GB4-1} (i),
we obtain $|\mathcal{T}(x)|_{\mathcal{B}_{d,\lambda}}\leq \beta_0$.

Recall that $\beta_1\geq x^{*}/(x^{*}-Mx^{*}e^{x^{*}}-M_1r_0e^{x^{*}})>0$.
Then $(Mx^{*}e^{x^{*}}+M_1r_0e^{x^{*}})/x^{*}\leq 1-1/\beta_1$, together with (\ref{x-star}), yields that
\begin{eqnarray*}
1+(Me^{r\lambda}+M_1e^{r\lambda}/\lambda)\beta_1
< 1+\beta_1(Mx^{*}e^{x^{*}}+M_1r_0e^{x^{*}})/x^{*}\leq 1+\beta_1(1-1/\beta_1)=\beta_1.
\end{eqnarray*}
It follows from Lemma \ref{lm-GB4-1} (ii) that
$|\mathcal{T}(x)|_{\mathcal{B}_{d,\lambda},1}\leq \beta_1$.

For $m=2,...,k+1$, note that $0<r< x^{*}(1-Me^{mx^{*}})\beta_{m}/A_{m}$,
then
\begin{eqnarray*}
M\beta_{m}e^{mr\lambda}+A_m/\lambda=M\beta_{m}e^{mx^{*}}+A_mr/x^{*}
\leq M\beta_{m}e^{mx^{*}}+A_mx^{*}(1-Me^{mx^{*}})\beta_{m}/(A_{m}x^{*})
\leq \beta_{m}.
\end{eqnarray*}
Then by Lemmas \ref{lm-GB4-1} (iii) and \ref{lm-GB4-2},
we have $|\mathcal{T}(x)|_{\mathcal{B}_{d,\lambda},m}\leq \beta_m$ for each $m=2,...,k+1$.
Thus, the operator $\mathcal{T}$  maps $\mathcal{B}_{d,\lambda}$ to itself.
\vskip 0.2cm
\noindent{\bf Step (iv).}
We finally prove that $\mathcal{T}$ is a contraction.
For any $x, y \in \mathcal{B}_{d,\lambda}$ and $(t,\xi) \in \mathbb{R}\times V^{0}_d$, we observe that
\begin{equation}
\begin{split}
|\mathcal{T}(x)(t,\xi)-\mathcal{T}(y)(t,\xi)|
\leq&\, |L(t)||x_t-y_t|+ |\int_0^{t} |F(s,x_s)-F(s,y_s)|ds|\\
\leq&\, M |x-y|_{\mathcal{B}_{d,\lambda}}e^{\lambda(r+|t|)}+|\int_0^{t} M_1|x_s-y_s|ds|\\
\leq&\, M |x-y|_{\mathcal{B}_{d,\lambda}}e^{\lambda(r+|t|)}+|\int_0^{t} M_1|x-y|_{\mathcal{B}_{d,\lambda}}e^{\lambda(r+|s|)}ds|\\
\leq&\, (Me^{r\lambda}+M_1e^{r\lambda}/\lambda)|x-y|_{\mathcal{B}_{d,\lambda}}e^{\lambda|t|}.\label{T-contr}
\end{split}
\end{equation}
Then by (\ref{x-star}) and (\ref{T-contr}),
\begin{eqnarray*}
|\mathcal{T}(x)-\mathcal{T}(y)|_{\mathcal{B}_{d,\lambda}}\leq((Mx^{*}e^{x^{*}}+M_1r_0e^{x^{*}})/x^{*})|x-y|_{\mathcal{B}_{d,\lambda}},
\end{eqnarray*}
which together with (\ref{x-star}) yields that $\mathcal{T}$ is a contraction.
By the Contraction Mapping Principle, $\mathcal{T}$ has a unique fixed point in the complete metric space $(\mathcal{B}_{d,\lambda}, \rho)$.
Assume that $\Psi$ is the unique fixed point of $\mathcal{T}$ in $(\mathcal{B}_{d,\lambda}, \rho)$.
By (\ref{map-1}), we can check that
$\Psi$ satisfies equation (\ref{NA-NDE}) and $\Psi(0,\xi)=\xi+L(0)\Psi_0$.
Therefore, Theorem \ref{thm-1} is proved.
\end{proof}

\begin{rmk}\label{rk-unique}
Assume that {\bf (H1)} holds.
By the similar method used in the proof of Theorem \ref{thm-1},
we can check that for each $\xi\in \mathbb{R}^{n}$,
there exists a unique solution $y$ of equation (\ref{NA-NDE}) satisfying
that $y$ is defined on $\mathbb{R}$, $\xi=y(0)-L(0)y_{0}$
and $\sup_{t\leq 0}|y(t)|e^{\lambda t}<+\infty$.
More precisely,
$y(t)=\Psi(t,\xi)$ for $t\in \mathbb{R}$, where $\Psi$ is defined in Theorem \ref{thm-1}.
In fact, to obtain this result,  we only need to consider the case $k=0$ in the proof of Theorem \ref{thm-1}.
\end{rmk}

\section{Proof of Theorem \ref{thm-1-2}}
\label{sec-pf-thm-1-2}

For any fixed $\phi\in \mathcal{C}$,
let $x(t)=x(t;0,\phi)$ for $t\in [-r,+\infty)$ denote the solution of equation (\ref{NA-NDE}) with $x_0=\phi$.
Let the constant $\lambda$ is defined as in {\bf (H2)} in this section.
Let $\mathcal{S}_{\lambda}$ denote the set of continuous maps $y:\mathbb{R}\to \mathbb{R}^n$ with
\begin{eqnarray*}
|y-x|_{\mathcal{S}_{\lambda},1}:=\sup_{t\geq 0}|y(t)-x(t)|e^{\lambda t}<+\infty \ \mbox{ and } \
|y|_{\mathcal{S}_{\lambda},2}:= \sup_{t\leq 0}|y(t)|e^{\lambda t}<+\infty.
\end{eqnarray*}
We define a map $\tilde{\rho}:\mathcal{S}_{\lambda}\times \mathcal{S}_{\lambda}\to \mathbb{R}$ in the form
$
\tilde{\rho}(f,g):=\sup_{t\in \mathbb{R}}|f(t)-g(t)|e^{\lambda t}
$
for any $f, g$ in $\mathcal{S}_{\lambda}$.
Note that for $f, g$ in $\mathcal{S}_{\lambda}$, we have
\begin{eqnarray*}
&&|f(t)-g(t)|e^{\lambda t}\leq |f(t)-x(t)|e^{\lambda t}+|g(t)-x(t)|e^{\lambda t} \ \ \ \mbox{ for } t\geq 0,\\
&&|f(t)-g(t)|e^{\lambda t}\leq |f(t)|e^{\lambda t}+|g(t)|e^{\lambda t} \ \ \ \mbox{ for } t\leq 0,
\end{eqnarray*}
then from the definition of the set $\mathcal{S}_{\lambda}$ it follows that the map $\tilde{\rho}$ is well defined.
Furthermore, we have the following result.

\begin{lm}
$(\mathcal{S}_{\lambda}, \tilde{\rho})$ is a complete metric space.
\end{lm}
\begin{proof}
Clearly, the map $\tilde{\rho}$ induces a metric on the set $\mathcal{S}_{\lambda}$.
To prove the completeness of the metric space $(\mathcal{S}_{\lambda}, \tilde{\rho})$,
take any Cauchy sequence $\{g_m\}_{m=1}^{+\infty}$ of $\mathcal{S}_{\lambda}$,
that is, for any $\epsilon>0$, there is a positive integer $N(\epsilon)$
such that for any positive integer $m, m'\geq N(\epsilon)$, $\tilde{\rho}(g_{m'},g_m)<\epsilon$,
which implies
\begin{eqnarray*}\label{sup-ym}
\sup_{t\geq 0}|g_{m'}(t)-g_{m}(t)|e^{\lambda t}<\epsilon, \ \ \ \sup_{t\leq 0}|g_{m'}(t)-g_{m}(t)|e^{\lambda t}<\epsilon.
\end{eqnarray*}
Similarly to Lemma \ref{lm-4-complete}, we observe that $\{g_m(t)e^{\lambda t}\}_{m=1}^{+\infty}$ is a Cauchy sequence in $C_b(\mathbb{R})$.
Let $\tilde{g}_0(t)$ be the limit of $g_{m}(t)e^{\lambda t}$
in $C_b(\mathbb{R})$ and $g_0(t):=\tilde{g}_0(t)e^{-\lambda t}$. Then we have
$\tilde{\rho}(g_{m},g_{0})\to 0$ as $m\to +\infty$.

By the definition of the set $\mathcal{S}_{\lambda}$, we have for sufficiently large $m$,
\begin{eqnarray*}
&&\sup_{t\geq 0}|g_{0}(t)-x(t)|e^{\lambda t}
\leq \sup_{t\geq 0}|g_{0}(t)-g_m(t)|e^{\lambda t}+\sup_{t\geq 0}|g_{m}(t)-x(t)|e^{\lambda t}<+\infty,\\
&&\sup_{t\leq 0}|g_{0}(t)|e^{\lambda t}
\leq \sup_{t\leq 0}|g_{0}(t)-g_m(t)|e^{\lambda t}+\sup_{t\leq 0}|g_{m}(t)|e^{\lambda t}<+\infty,
\end{eqnarray*}
which implies that $g_0\in \mathcal{S}_{\lambda}$.
Therefore, the proof is complete.
\end{proof}

Next we prove Theorem \ref{thm-1-2} by constructing a contraction operator
on the complete metric space $(\mathcal{S}_{\lambda}, \tilde{\rho})$,
and then  applying the statements in Remark \ref{rk-unique}.
The similar method was widely used to establish the existence of the desired solutions of differential equations
(see, for instance, \cite{Burton-85,Driver1976}).
\begin{proof}[Proof of Theorem \ref{thm-1-2}]
For any $y\in \mathcal{S}_{\lambda}$, we define the operator $\mathcal{Q}$:
\qquad \begin{eqnarray}\label{eq-def-T}
\mathcal{Q}(y)(t)
=\left\{
\begin{array}{ll}
L(t)y_t+x(t)-L(t)x_t-\int_{t}^{+\infty}(F(s,y_s)-F(s,x_s)) ds,
& t>0,
\\
x(0)-L(0)x_0-\int_{0}^{+\infty}(F(s,y_s)-F(s,x_s)) ds +L(t)y_t
\\
+\int_{0}^{t}F(s,y_s) ds,
& t\leq0.
\end{array}
\right.
\end{eqnarray}
For any $y\in \mathcal{S}_{\lambda}$ and $t_2\geq t_1\geq r$, we note that
\begin{eqnarray*}
\lefteqn{|\int_{t_1}^{t_2}(F(s,y_s)-F(s,x_s)) ds|}\\
&\leq&\!\!\!\!
|\int_{t_1}^{t_2} M_1|y_s-x_s| ds|\leq M_1|\int_{t_1}^{t_2}|y-x|_{\mathcal{S}_{\lambda},1}\, e^{\lambda(r-s)} ds|
\leq M_1|y-x|_{\mathcal{S}_{\lambda},1}\,e^{\lambda(r-t_1)}/\lambda.
\end{eqnarray*}
By the Cauchy Convergence Principle, we see that the integral $\int_{t}^{+\infty}(F(s,y_s)-F(s,x_s)) ds$ is well defined for $t\geq 0$.
Moreover, for $t \geq 0$ we find that
\begin{equation}
\begin{split}
|\int_{t}^{+\infty}(F(s,y_s)-F(s,x_s)) ds| \leq&\, M_1\int_{t}^{+\infty}|y_s-x_s| ds\leq M_1\int_{0}^{+\infty}|y_s-x_s| ds\\
=&\,
M_1\int_{0}^{r}|y_s-x_s| ds+M_1\int_{r}^{+\infty}|y_s-x_s| ds\\
\leq &\,
M_1r\left(\max_{t\in[-r,0]}|y(t)-x(t)|+\max_{t\in[0,r]}|y(t)-x(t)|\right)
\\
&\,+M_1|y-x|_{\mathcal{S}_{\lambda},1}/\lambda\\
\leq &\, M_1r\left(|y|_{\mathcal{S}_{\lambda},2}e^{r\lambda}+|\phi|+|y-x|_{\mathcal{S}_{\lambda},1}\right)+M_1|y-x|_{\mathcal{S}_{\lambda},1}/\lambda.
\label{est-Fy-Fx}
\end{split}
\end{equation}
Since both $x$ and $y$ are continuous on $[-r,+\infty)$ and $\mathbb{R}$, we have
\begin{equation}
\begin{split}
\mathcal{Q}(y)(0^+)=&\,L(0)y_0+x(0)-L(0)x_0-\int_{0}^{+\infty}(F(s,y_s)-F(s,x_s)) ds\label{Ty-0}\\
=&\,\mathcal{Q}(y)(0^-)=\mathcal{Q}(y)(0).
\end{split}
\end{equation}
Thus $T(y)$ is continuous at $t=0$. Furthermore, by the continuity of $x$ and $y$, $T(y)$ is continuous on $\mathbb{R}$.
As $t\geq r$, we have
\begin{eqnarray*}
|\mathcal{Q}(y)(t)-x(t)|\!\!\!&\leq&\!\!\! |L(t)(y_t-x_t)| +|\int_{t}^{+\infty}(F(s,y_s)-F(s,x_s)) ds|\\
\!\!\!&\leq&\!\!\! M|y_t-x_t|+|\int_{t}^{+\infty}M_1|y_s-x_s|ds|\\
\!\!\!&\leq&\!\!\! (M+ M_1/\lambda)|y-x|_{\mathcal{S}_{\lambda},1}\,e^{\lambda(r-t)},
\end{eqnarray*}
which implies $|\mathcal{Q}(y)-x|_{\mathcal{S}_{\lambda},1}<+\infty$.
As $t\leq 0$, by (\ref{est-Fy-Fx}), (\ref{Ty-0}) and the properties of the space $\mathcal{S}_{\lambda}$, we find that
\begin{equation}
\label{Ty-est}
\begin{split}
|\mathcal{Q}(y)(t)|
\leq&\,
|x(0)-L(0)x_0-\int_{0}^{+\infty}(F(s,y_s)-F(s,x_s)) ds|\\
&\,
+|L(t)y_t|+|\int_{0}^{t}(F(s,y_s)-F(s,0)) ds|+|\int_0^{t} |F(s,0)|ds|\\
\leq&\, |x(0)-L(0)x_0|+\int_{0}^{+\infty} M_1 |y_s-x_s| ds\\
&\,
+M|y|_{\mathcal{S}_{\lambda}, 2}e^{\lambda(r-t)}+|\int_{0}^{t}M_1|y|_{\mathcal{S}_{\lambda},2}e^{\lambda(r-s)} ds|+|\int_0^t M_0 e^{\lambda |s|}ds|\\
&\leq\,
(1+M)|\phi|+M_1r\left(|y|_{\mathcal{S}_{\lambda},2}e^{r\lambda}+|\phi|+|y-x|_{\mathcal{S}_{\lambda},1}\right)
+M_1|y-x|_{\mathcal{S}_{\lambda},1}/\lambda\\
&\,
+M|y|_{\mathcal{S}_{\lambda}, 2}e^{\lambda(r-t)}+M_1|y|_{\mathcal{S}_{\lambda}, 2}\,e^{\lambda(r-t)} /\lambda+M_0e^{-\lambda t}/\lambda\\
=&\,
\left\{(1+M+M_1r)|\phi|+M_0/\lambda+M_1(r+1/\lambda)|y-x|_{\mathcal{S}_{\lambda},1}\right.\\
&\,\left.+(M_1re^{r\lambda}+Me^{r\lambda}+M_1e^{r\lambda} /\lambda)|y|_{\mathcal{S}_{\lambda},2}\right\}e^{-\lambda t},
\end{split}
\end{equation}
which yields that $|\mathcal{Q}(y)|_{\mathcal{S}_{\lambda},2}<+\infty$.
Therefore, $\mathcal{Q}$ maps $\mathcal{S}_{\lambda}$ into itself.

To prove that $\mathcal{Q}$ is a contraction,
for any $y,z \in \mathcal{S}_{\lambda}$ and $t\in \mathbb{R}$, we see that
\begin{eqnarray*}
|\mathcal{Q}(y)(t)-\mathcal{Q}(z)(t)|
\!\!\!&\leq&\!\!\!\!
|L(t)(y_t-z_t)|+|\int_{t}^{+\infty}(F(s,y_s)-F(s,z_s))ds|\\
\!\!\!&\leq&\!\!\!\!
Md(x,y)e^{\lambda(r-t)}+|\!\int_{t}^{+\infty}\!\!\! M_1\tilde{\rho}(x,y)e^{\lambda(r-s)}ds|\\
\!\!\!&\leq&\!\!\!\!
\left(Me^{r\lambda}+M_1e^{r\lambda}/\lambda\right)e^{-\lambda t} \tilde{\rho}(x,y),
\end{eqnarray*}
together with (\ref{x-star}), yields
$$
|\mathcal{Q}(y)(t)-\mathcal{Q}(z)(t)|e^{\lambda t} \leq  \left((Mx^{*}e^{x^{*}}+M_1r_0e^{x^{*}})/x^{*}\right)\tilde{\rho}(x,y)<\tilde{\rho}(x,y).
$$
Thus $\mathcal{Q}$ is a contraction.

Applying the Contraction Mapping Principle, $\mathcal{Q}$ has a unique fixed point in the complete metric space $(\mathcal{S}_{\lambda}, \tilde{\rho})$,
denoted by $y$. By (\ref{eq-def-T}) we can check that
$y$ satisfies equation (\ref{NA-NDE}) for $t\in \mathbb{R}$
and by (\ref{Ty-est}) we have $|y|_{\mathcal{S}_{\lambda},2}<+\infty$.
Take $\xi=x(0)-L(0)x_0-\int_{0}^{+\infty}(F(s,y_s)-F(s,x_s)) ds\in\mathbb{R}^{n}$.
Using (\ref{eq-def-T}) again, we find that $\xi=y(0)-L(0)y_0$.
Furthermore, by Remark \ref{rk-unique} we have the fixed point $y(t)=\Psi(t,\xi)$ for $t\in \mathbb{R}$.
In the end, following the definition of $\mathcal{S}_{\lambda}$,
we obtain $\sup_{t\geq0}|x(t; 0,\phi)-y(t)|e^{\lambda t}<+\infty$.
Therefore, the proof of Theorem \ref{thm-1-2} is complete.
\end{proof}

\section{Proof of Theorem \ref{thm-2}}
\label{sec-pf-thm3}
For any fixed constant $d>1$, set $V_{0}:=V_{d}^0=\{\xi\in \mathbb{R}^n: |\xi|<d\}$ and $V_{1}:=R^n\backslash V_d^0$.
Let the constants $r_0$, $\delta$, $\lambda$ be given in {\bf (H2)}
and $\Omega$ denote the interval $(0,\delta)$.
For each $\gamma\in\{0,1\}$,
let $\mathcal{E}_{\gamma,\lambda}$ be the set of continuous maps
from $\mathbb{R}\times\Omega\times V_{\gamma}$ to $\mathbb{R}^n$,
and satisfy for each $x\in \mathcal{E}_{\gamma,\lambda}$,
$x$ is $C^k$ in $(t,r) \in \mathbb{R}\times\Omega$ for each $\xi\in V_{\gamma}$ and
for some constants $\varepsilon_{j}>0$, $j=0,1,...,k+1$,
\begin{eqnarray}
|x|_{\mathcal{E}_{\gamma,\lambda}}\!\!\!&:=&\!\!\!
\sup_{(t,r,\xi)\in \mathbb{R}\times\Omega\times V_{\gamma}} |x(t,r,\xi)|(e^{\lambda|t|}|\xi|^{\gamma})^{-1}
\leq \varepsilon_0,\label{2-norm-x}\\
|x|_{\mathcal{E}_{\gamma,\lambda},j}\!\!\!&:=&\!\!\!
\sup_{(t,r,\xi)\in \mathbb{R}\times\Omega\times V_{\gamma}} |D^jx(t,r,\xi)|(e^{\lambda|t|}|\xi|^{\gamma})^{-j}
\leq \varepsilon_j, \ \ \ j=1,2,...,k,\label{2-bd-xj}\\
|x|_{\mathcal{E}_{\gamma,\lambda},k+1}\!\!\!&:=&\!\!\!\!\!\!\!\!
\!\!\!\sup_{\xi\in V_{\gamma},\, (t_1,r_1)\neq (t_2,r_2), (t_i,r_i)\in\mathbb{R}\times\Omega}\!\!\!\!\!\!\!\!\!
\frac{|D^k x(t_1,r_1,\xi)-D^k x(t_2,r_2,\xi)|}{|(t_1,r_1)-(t_2,r_2)|}(e^{\lambda|t_{*}|}|\xi|^{\gamma})^{-(k+1)}
\!\! \leq \varepsilon_{k+1},\label{2-bd-lip-x}
\end{eqnarray}
where $|t_*|=\max\{|t_1|,|t_2|\}$ and for simplicity,
in this section $D^{j}x$ denote the {\it j-}th derivative of $x$ with respect to $(t,r)$.
\begin{lm}\label{lm-5-complete}
For each $\gamma \in \{0,1\}$,
the metric space $(\mathcal{E}_{\gamma,\lambda}, \rho_{\gamma})$ is complete,
where $\rho_{\gamma}(x,y)=|x-y|_{\mathcal{E}_{\gamma,\lambda}}$ for any $x, y\in \mathcal{E}_{\gamma,\lambda}$.
\end{lm}
\begin{proof}
Clearly, for each $\gamma \in \{0,1\}$, $\rho_{\gamma}$ is well defined
and induces a metric for the set $\mathcal{E}_{\gamma,\lambda}$.
To prove the completeness of the metric space $(\mathcal{E}_{\gamma,\lambda}, \rho_{\gamma})$,
take any Cauchy sequence $\{g_m\}_{m=1}^{+\infty}$ of $\mathcal{E}_{\gamma,\lambda}$,
that is, for any $\epsilon>0$,
there is a positive integer $N(\epsilon)$ such that for any $m, m' \geq N(\epsilon)$,
\begin{eqnarray}\label{Cauchy-seq-2}
\rho_{\gamma}(g_{m'},g_m)=\sup_{(t,r,\xi)\in \mathbb{R}\times\Omega\times V_{\gamma}}
|g_{m'}(t,r,\xi)-g_m(t,r,\xi)|(e^{\lambda|t|}|\xi|^{\gamma})^{-1} <\!\epsilon.
\end{eqnarray}
Similarly to Lemma \ref{lm-4-complete}, there exists a continuous map $g_0$ from $\mathbb{R}\times\Omega\times V_{\gamma}$ to $\mathbb{R}^n$ such that $\rho_{\gamma}(g_{m},g_{0})\to 0$ as $m\to+\infty$.

Next we claim that $g_0 \in \mathcal{E}_{\gamma,\lambda}$.
For any $g\in \mathcal{E}_{\gamma,\lambda}$ and any $t_0>0$,
let $\xi\in V_{\gamma}$ be fixed and $\widetilde{g}(t,r)=g(t,r,\xi)$ for $(t,r)\in(-t_0,t_0)\times(0,\delta)$.
Then we see that
$|\widetilde{g}(t,r)|=|g(t,r,\xi)|
\leq |g|_{\mathcal{E}_{\gamma,\lambda}} e^{\lambda|t_0|}|\xi|^{\gamma}.$
Recall that $\rho_{\gamma}(g_{m},g_{0})\to 0$ as $m\to+\infty$.
Then we have $|\widetilde{g}_m-\widetilde{g}_0|_{\infty}\to0$ as $m\to+\infty$.
From (\ref{2-norm-x})-(\ref{2-bd-lip-x}) it follows that
$\{\widetilde{g}_m\}_{m=1}^{+\infty}\subset C_{\varepsilon(t_0,\xi)}^{k,1}$,
where $\varepsilon(t_0,\xi)=\max\{\varepsilon_{0}e^{\lambda|t_0|}|\xi|^{\gamma},
\varepsilon_j(e^{\lambda|t_0|}|\xi|^{\gamma})^{j} \mbox{ for }j=1,2,...,k+1\}$.
Hence, by Lemma \ref{Henry-lm}, we obtain
$\widetilde{g}_0\in C_{\varepsilon(t_0,\xi)}^{k,1}$
and $D^j\widetilde{g}_m(t,r)\to D^j\widetilde{g}_0(t,r)$ as $m\to +\infty$
for each $(t,r)\in (-t_0,t_0)\times(0,\delta)$ and $j=1,2,...,k$.
Note that the arbitrariness of $t_0$. Using (\ref{2-norm-x})-(\ref{2-bd-lip-x}) again, we obtain that $g_0\in \mathcal{E}_{\gamma,\lambda}$.
Thus, the claim is true and Lemma \ref{lm-5-complete} is established.
\end{proof}

On each $\mathcal{E}_\gamma$, we define a map $\mathcal{F}$ in the form
\begin{eqnarray}\label{map-2}
\mathcal{F}(x)(t,r,\xi):=\xi+Ax(t-r,r,\xi)+\int_0^{t}f(x(s,r,\xi),x(s-r,r,\xi))ds.
\end{eqnarray}
To  prove Theorem \ref{thm-2}, we first make some preparations.

\begin{lm}\label{lm-GB5-1}
Let $x\in \mathcal{E}_{\gamma,\lambda}$ and $g(t,r,\xi)=x(t-r,r,\xi)$ for $(t,r,\xi)\in \mathbb{R}\times\Omega\times V_{\gamma}$.
Then for each $\boldsymbol{\nu}=(\nu_1,\nu_2)\in\mathbb{N}_0^2$ with $1\leq|\boldsymbol{\nu}|\leq k$,
the following results hold:
\begin{itemize}
\item[(i)]
For any $(t,r,\xi) \in \mathbb{R}\times\Omega\times V_{\gamma}$,
\begin{eqnarray}\label{deriv-xtr}
\frac{\partial^{\nu_1+\nu_2}}{\partial t^{\nu_1}\partial r^{\nu_2} }g(t,r,\xi)
=\sum_{j=0}^{\nu_2}\frac{(-1)^{j}v_2!}{j!(\nu_2-j)!}
\frac{\partial^{\nu_1+\nu_2}}{\partial y_1^{\nu_1+j}\partial y_2^{\nu_2-j}} x(t-r,r,\xi),
\end{eqnarray}
and
$
|\frac{\partial^{\nu_1+\nu_2}}{\partial t^{\nu_1}\partial r^{\nu_2} }g(t,r,\xi)|
\leq \varepsilon_{|\boldsymbol{\nu}|}S_{\boldsymbol{\nu}}\left(e^{\lambda|t|}|\xi|^{\gamma}\right)^{|\boldsymbol{\nu}|},
$
where the constant $S_{\boldsymbol{\nu}}=2^{\nu_2}e^{|\boldsymbol{\nu}|x^{*}}$
and $x^{*}$ is defined in {\bf (H2)}.

\item[(ii)]
For any $(t_1,r_1,\xi),\,(t_2,r_2,\xi)$ in $\mathbb{R}\times\Omega\times V_{\gamma}$ with $(t_1,r_1)\neq(t_2,r_2)$
and $|\boldsymbol{\nu}|=k$,
\begin{eqnarray*}
\!\!\! |\frac{\partial^{k}}{\partial t^{\nu_1}\partial r^{\nu_2} }g(t_1,r_1,\xi)
-\frac{\partial^{k}}{\partial t^{\nu_1}\partial r^{\nu_2} }g(t_2,r_2,\xi)|
\leq \varepsilon_{k+1}S_{(\nu_1,\nu_2+1)}\left(e^{\lambda|t_{*}|}|\xi|^{\gamma}\right)^{k+1}
|(t_1,r_1)-(t_2,r_2)|,
\end{eqnarray*}
where $|t_{*}|=\max\{|t_1|,|t_2|\}$.

\item[(iii)]
For any $(t,r,\xi) \in \mathbb{R}\times\Omega\times V_{\gamma}$,
$
|\frac{\partial^{\nu_1+\nu_2}}{\partial t^{\nu_1}\partial r^{\nu_2} }f(x(t,r,\xi),x(t-r,r,\xi))|
\leq T_{\boldsymbol{\nu}}\left(e^{\lambda|t|}|\xi|^{\gamma}\right)^{|\boldsymbol{\nu}|},
$
where the constant
$$
T_{\boldsymbol{\nu}}= \sum_{1\leq |\boldsymbol{\omega}|\leq |\boldsymbol{\nu}|}M_{|\boldsymbol{\omega}|}
\sum_{s=1}^{|\boldsymbol{\nu}|}\sum_{p_s(\boldsymbol{\nu},\boldsymbol{\omega})}(\boldsymbol{\nu}!)
\prod_{j=1}^{s}\frac{1}{(\boldsymbol{k_j}!)(\boldsymbol{l_j}!)^{|\boldsymbol{k_j}|}}
\varepsilon_{|\boldsymbol{l_j}|}^{k_{j1}}\left(\varepsilon_{|\boldsymbol{l_j}|}S_{\boldsymbol{l_j}}\right)^{k_{j2}},
$$
and $\boldsymbol{\omega}=(\omega_1,\omega_2)$, $\boldsymbol{k_j}=(k_{j1},k_{j2})$, $\boldsymbol{l_j}=(l_{j1},l_{j2})$, $p_s(\boldsymbol{\nu},\boldsymbol{\omega})$
are defined in Lemma \ref{partial-devt}.
\end{itemize}
\end{lm}

\begin{proof}
For each $\boldsymbol{\nu}=(\nu_1,\nu_2)\in\mathbb{N}_0^2$  with $1\leq|\boldsymbol{\nu}|\leq k$
and any $(t,r,\xi) \in \mathbb{R}\times\Omega\times V_{\gamma}$,
we note that
\begin{eqnarray*}
\frac{\partial^{\nu_1+\nu_2}}{\partial t^{\nu_1}\partial r^{\nu_2} }g(t,r,\xi)=
\frac{\partial^{\nu_2}}{\partial r^{\nu_2}}\left(\frac{\partial^{\nu_1}}{\partial t^{\nu_1}}x(t-r,r,\xi) \right)
=\sum_{j=0}^{\nu_2}\frac{(-1)^{j}v_2!}{j!(\nu_2-j)!}
\frac{\partial^{\nu_1+\nu_2}}{\partial y_1^{\nu_1+j}\partial y_2^{\nu_2-j}} x(t-r,r,\xi),
\end{eqnarray*}
which implies that
\begin{eqnarray*}
|\frac{\partial^{\nu_1+\nu_2}}{\partial t^{\nu_1}\partial r^{\nu_2}}g(t,r,\xi)|
\leq
\sum_{j=0}^{\nu_2}\frac{v_2!}{j!(\nu_2-j)!}\varepsilon_{|\boldsymbol{\nu}|}
\!\left(e^{\lambda(r+|t|)}|\xi|^{\gamma}\right)\!^{|\boldsymbol{\nu}|}
\leq
2^{\nu_2}e^{|\boldsymbol{\nu}|r\lambda}\varepsilon_{|\boldsymbol{\nu}|}\!\left(e^{\lambda|t|}|\xi|^{\gamma}\right)\!^{|\boldsymbol{\nu}|}.
\end{eqnarray*}
In view of $0<r<\delta$ and $\delta\lambda=x^{*}$, result (i) is proved.

For any $(t_1,r_1,\xi),\,(t_2,r_2,\xi)$ in $\mathbb{R}\times\Omega\times V_{\gamma}$ with $(t_1,r_1)\neq(t_2,r_2)$
and $|\boldsymbol{\nu}|=k$, by (\ref{deriv-xtr}) we have
\begin{eqnarray*}
\lefteqn{|\frac{\partial^{k}}{\partial t^{\nu_1}\partial r^{\nu_2} }g(t_1,r_1,\xi)
-\frac{\partial^{k}}{\partial t^{\nu_1}\partial r^{\nu_2} }g(t_2,r_2,\xi)|}\\
\!\!\!&\leq&\!\!\!
\sum_{j=0}^{\nu_2}\frac{\nu_2!}{j!(\nu_2-j)!}
|\frac{\partial^{k}}{\partial y_1^{\nu_1+j}\partial y_2^{\nu_2-j}}x(t_1-r_1,r_1,\xi)
-\frac{\partial^{k}}{\partial y_1^{\nu_1+j}\partial y_2^{\nu_2-j}}x(t_2-r_2,r_2,\xi)|\\
\!\!\!&\leq&\!\!\!
\sum_{j=0}^{\nu_2}\frac{\nu_2!}{j!(\nu_2-j)!}\varepsilon_{k+1}
\left(e^{\lambda(\delta+|t_{*}|)}|\xi|^{\gamma}\right)^{k+1}\max\{|t_1-t_2-r_1+r_2|,|r_1-r_2|\}\\
\!\!\!&\leq&\!\!\varepsilon_{k+1} S_{(\nu_1,\nu_2+1)}\left(e^{\lambda|t_{*}|}|\xi|^{\gamma}\right)^{k+1}
\max\{|t_1-t_2|,|r_1-r_2|\}.
\end{eqnarray*}
Then result (ii) is proved.

By Lemma \ref{partial-devt}, we obtain
\begin{eqnarray*}
\lefteqn{|\frac{\partial^{\nu_1+\nu_2}}{\partial t^{\nu_1}\partial r^{\nu_2} }f(x(t,r,\xi),x(t-r,r,\xi))|}\\
\!\!\!\!&\leq&\!\!\!\!\!\!\!
\sum_{1\leq |\boldsymbol{\omega}|\leq |\boldsymbol{\nu}|}\!
|\frac{\partial^{|\boldsymbol{\omega}|}}{\partial y_1^{\omega_1}\partial y_2^{\omega_2}}f|
\sum_{s=1}^{|\boldsymbol{\nu}|}\!\!\sum_{p_s(\boldsymbol{\nu},\boldsymbol{\omega})}(\boldsymbol{\nu}!)\\
\!\!\!\!& &\!\!\! \times
\prod_{j=1}^{s}\frac{1}{(\boldsymbol{k_j}!)(\boldsymbol{l_j}!)^{|\boldsymbol{k_j}|}}
|\frac{\partial^{|\boldsymbol{l_j}|}}{\partial t^{l_{j1}}\partial r^{l_{j2}} }x(t,r,\xi)|^{k_{j1}}
|\frac{\partial^{|\boldsymbol{l_j}|}}{\partial t^{l_{j1}}\partial r^{l_{j2}} }x(t-r,r,\xi)|^{k_{j2}}\\
\!\!\!&\leq&\!\!\!\!\! \sum_{1\leq |\boldsymbol{\omega}|\leq |\boldsymbol{\nu}|}M_{|\boldsymbol{\omega}|}
\sum_{s=1}^{|\boldsymbol{\nu}|}\sum_{p_s(\boldsymbol{\nu},\boldsymbol{\omega})}(\boldsymbol{\nu}!)
\prod_{j=1}^{s}\frac{1}{(\boldsymbol{k_j}!)(\boldsymbol{l_j}!)^{|\boldsymbol{k_j}|}} \left(\varepsilon_{|\boldsymbol{l_j}|}\left(e^{\lambda|t|}|\xi|^{\gamma}\right)^{|\boldsymbol{l_j}|}\right)^{k_{j1}}
\!\!\!\left(\varepsilon_{|\boldsymbol{l_j}|}S_{\boldsymbol{l_j}}\left(e^{\lambda|t|}|\xi|^{\gamma}\right)^{|\boldsymbol{l_j}|}\right)^{k_{j2}},
\end{eqnarray*}
together with $\sum_{j=1}^{s} |\boldsymbol{k_j}|\boldsymbol{l_j}=\boldsymbol{\nu}$, yields result (iii).
Therefore, the proof is complete.
\end{proof}

\begin{lm}\label{lm-GB5-2}
Let $x\in \mathcal{E}_{\gamma,\lambda}$ and $f$ satisfy the conditions in {\bf (H2)}.
Then for any $(t,r_1,\xi)$ and $(t,r_2,\xi)$ in $\mathbb{R}\times\Omega\times V_{\gamma}$ with $r_1\neq r_2$,
\begin{eqnarray*}
&&\!\!\!\!\!\!\!\!\!
|\frac{\partial^{k}}{\partial r^{k}} f(x(t,r_1,\xi),x(t-r_1,r_1,\xi))
-\!\frac{\partial^{k}}{\partial r^{k}} f(x(t,r_2,\xi),x(t-r_2,r_2,\xi))|
\leq T_{(0,k+1)}\!\left(e^{\lambda|t|}|\xi|^{\gamma}\right)^{k+1}\!\!|r_1-r_2|,
\end{eqnarray*}
where
\begin{eqnarray*}
T_{(0,k+1)}\!\!\!&=&\!\!\!
(k!)\!\sum_{1\leq|\boldsymbol{\omega}|\leq k}
M_{|\boldsymbol{\omega}|+1}\varepsilon_{1}S_{(0,1)}
\sum_{s=1}^{k}\sum_{p_s(\boldsymbol{\nu},\boldsymbol{\omega})}
\prod_{j=1}^{s}
\frac{\left(\varepsilon_{|\boldsymbol{l_j}|}S_{\boldsymbol{l_j}}\right)^{|\boldsymbol{k_j}|}}
{(\boldsymbol{k_j}!)(\boldsymbol{l_j}!)^{|\boldsymbol{k_j}|}}
+(k!)\!\sum_{1\leq|\boldsymbol{\omega}|\leq k}M_{|\boldsymbol{\omega}|}
\sum_{s=1}^{k}\sum_{p_s(\boldsymbol{\nu},\boldsymbol{\omega})}  \mathcal{J}_s,\\
\mathcal{K}_{j}\!\!\!&=&\!\!\!
\sum_{i=1}^{2n}\!k_{j,i} \varepsilon_{|\boldsymbol{l_j}|+1}S_{(l_{j1},l_{j2}+1)} \left(\varepsilon_{|\boldsymbol{l_j}|}S_{\boldsymbol{l_j}}\right)^{k_{j,i}-1}
\prod_{m=0}^{i-1} \left(\varepsilon_{|\boldsymbol{l_j}|}S_{\boldsymbol{l_j}}\right)^{k_{j,m}}
\prod_{m=i+1}^{2n+1}\left(\varepsilon_{|\boldsymbol{l_j}|}S_{\boldsymbol{l_j}}\right)^{k_{j,m}},\\
\mathcal{J}_{s}\!\!\!&=&\!\!\!
\sum_{j=1}^{s}\frac{\mathcal{K}_{j}}
{(\boldsymbol{k_j}!)(\boldsymbol{l_j}!)^{|\boldsymbol{k_j}|}}
\prod_{i=0}^{j-1}
\frac{\left(\varepsilon_{|\boldsymbol{l_i}|}S_{\boldsymbol{l_i}}\right)^{|\boldsymbol{k_i}|}}
{(\boldsymbol{k_i}!)(\boldsymbol{l_i}!)^{|\boldsymbol{k_i}|}}
\prod_{i=j+1}^{s+1}
\frac{\left(\varepsilon_{|\boldsymbol{l_i}|}S_{\boldsymbol{l_i}}\right)^{|\boldsymbol{k_i}|}}
{(\boldsymbol{k_i}!)(\boldsymbol{l_i}!)^{|\boldsymbol{k_i}|}},
\end{eqnarray*}
$\boldsymbol{\omega}=(\omega_1,...,\omega_{2n})$, $\boldsymbol{\nu}=(0,k)$,
$\boldsymbol{k_j}=(k_{j,1},...,k_{j,2n})$, $\boldsymbol{l_j}=(l_{j,1}, l_{j,2})$,
$\boldsymbol{k_0}=\boldsymbol{k_{s+1}}=(0,...,0)$ and $k_{j,0}=k_{j,2n+1}=0$.
\end{lm}

\begin{proof}
Let $x(t, r, \xi)=(x_1(t, r, \xi),...,x_n(t, r, \xi))^{T}\in \mathbb{R}^{n}$ and
${\bf g}(t,r,\xi)=(g_1(t,r,\xi),...,g_{2n}(t,r,\xi))
=(x_1(t,r,\xi),...,x_n(t,r,\xi),x_1(t-r,r,\xi),...,x_n(t-r,r,\xi))^{T}$.
To simplify the notations,
we also denote $f(r):=f(x(t,r,\xi),x(t-r,r,\xi))$ and ${\bf g}(r):={\bf g}(t,r,\xi)$.
By Appendix C, we obtain
\begin{equation}\label{k-lip}
\begin{split}
\lefteqn{ |\frac{\partial^{k}}{\partial r^{k}} f(x(t,r_1,\xi),x(t-r_1,r_1,\xi))
-\frac{\partial^{k}}{\partial r^{k}} f(x(t,r_2,\xi),x(t-r_2,r_2,\xi))|} \\
\leq&\,
\sum_{1\leq|\boldsymbol{\omega}|\leq k}\!\!
|f_{\boldsymbol{\omega}}(r_1)-f_{\boldsymbol{\omega}}(r_2)|
\sum_{s=1}^{k}\!\sum_{p_s(\boldsymbol{\nu},\boldsymbol{\omega})}
(k!)\prod_{j=1}^{s}
\frac{|({\bf g}_{\boldsymbol{l_j}}(r_1))^{\boldsymbol{k_j}}|}{(\boldsymbol{k_j}!)(\boldsymbol{l_j}!)^{|\boldsymbol{k_j}|}}
+\!\!\!\sum_{1\leq|\boldsymbol{\omega}|\leq k}\!\! |f_{\boldsymbol{\omega}}(r_2)|
\sum_{s=1}^{k}\!\sum_{p_s(\boldsymbol{\nu},\boldsymbol{\omega})}(k!) \Delta_{s},
\end{split}
\end{equation}
where
\begin{eqnarray*}
\Delta_{s}\!\!\!&:=&\!\!\!
\sum_{j=1}^{s}\left(\frac{\Theta_{j}}
{(\boldsymbol{k_j}!)(\boldsymbol{l_j}!)^{|\boldsymbol{k_j}|}}\right)
\prod_{i=0}^{j-1}
\frac{|({\bf g}_{\boldsymbol{l_i}}(r_2))^{\boldsymbol{k_i}}|}{(\boldsymbol{k_i}!)(\boldsymbol{l_i}!)^{|\boldsymbol{k_i}|}}
\prod_{i=j+1}^{s+1}
\frac{|({\bf g}_{\boldsymbol{l_i}}(r_1))^{\boldsymbol{k_i}}|}{(\boldsymbol{k_i}!)(\boldsymbol{l_i}!)^{|\boldsymbol{k_i}|}},\\
\Theta_{j}\!\!\!&:=&\!\!\!
\sum_{i=1}^{2n}\Theta_{i,j} \prod_{m=0}^{i-1} |g_{\boldsymbol{l_j}}^{(m)}(r_2)|^{k_{j,m}}
\prod_{m=i+1}^{2n+1}|g_{\boldsymbol{l_j}}^{(m)}(r_1)|^{k_{j,m}},\\
\Theta_{i,j}\!\!\!&:=&\!\!\!
|g^{(i)}_{\boldsymbol{l_j}}(r_1)-g^{(i)}_{\boldsymbol{l_j}}(r_2)|
\sum_{m=0}^{k_{j,i}-1}|g^{(i)}_{\boldsymbol{l_j}}(r_1)|^{m}\,|g^{(i)}_{\boldsymbol{l_j}}(r_2)|^{k_{j,i}-m-1},
\end{eqnarray*}
$\boldsymbol{k_0}=\boldsymbol{k_{s+1}}=(0,...,0)$ and $k_{j,0}=k_{j,2n+1}=0$.
We first note that for any $1\leq i\leq n$,
\begin{eqnarray*}
\Theta_{i,j}
\!\!\!&\leq&\!\!\!
\varepsilon_{|\boldsymbol{l_j}|+1} \left(e^{\lambda|t|}|\xi|^{\gamma}\right)^{|\boldsymbol{l_j}|+1}\!\!|r_1-r_2|
\sum_{m=0}^{k_{j,i}-1}
\left(\varepsilon_{|\boldsymbol{l_j}|} \left(e^{\lambda|t|}|\xi|^{\gamma}\right)^{|\boldsymbol{l_j}|}\right)^{k_{j,i}-1}\nonumber\\
\!\!\!&=&\!\!\!  k_{j,i}\, \varepsilon_{|\boldsymbol{l_j}|+1}
\varepsilon_{|\boldsymbol{l_j}|}^{k_{j,i}-1}
\left(e^{\lambda|t|}|\xi|^{\gamma}\right)^{k_{j,i}|\boldsymbol{l_j}|+1}|r_1-r_2|.
\end{eqnarray*}
Similarly,  we have
$\Theta_{i,j}\leq
k_{j,i} \varepsilon_{|\boldsymbol{l_j}|+1}S_{(l_{j1},l_{j2}+1)}
\left(\varepsilon_{|\boldsymbol{l_j}|}S_{\boldsymbol{l_j}}\right)^{k_{j,i}-1}
\left(e^{\lambda|t|}|\xi|^{\gamma}\right)^{k_{j,i}|\boldsymbol{l_j}|+1}|r_1-r_2|$ for $n+1\leq i \leq 2n$.
In view of $1\leq S_{(l_{j1},l_{j2}+1)}$
and $1\leq S_{\boldsymbol{l_j}}$, applying Lemma \ref{lm-GB5-1} (i), we obtain
\begin{eqnarray*}
\Theta_{j}
\!\!\!&\leq&\!\!\!
\sum_{i=1}^{2n}\!k_{j,i} \varepsilon_{|\boldsymbol{l_j}|+1}S_{(l_{j1},l_{j2}+1)}
\left(\varepsilon_{|\boldsymbol{l_j}|}S_{\boldsymbol{l_j}}\right)^{k_{j,i}-1}
\!\left(e^{\lambda|t|}|\xi|^{\gamma}\right)^{k_{j,i}|\boldsymbol{l_j}|+1}
\!\prod_{m=0}^{i-1} \left(\varepsilon_{|\boldsymbol{l_j}|}S_{\boldsymbol{l_j}}\left(e^{\lambda|t|}|\xi|^{\gamma}\right)^{|\boldsymbol{l_j}|}\right)^{k_{j,m}}
\nonumber\\
\!\!\!& &\!\!\!
\times \prod_{m=i+1}^{2n+1}\left(\varepsilon_{|\boldsymbol{l_j}|}S_{\boldsymbol{l_j}}
\left(e^{\lambda|t|}|\xi|^{\gamma}\right)^{|\boldsymbol{l_j}|}\right)^{k_{j,m}}
|r_1-r_2|
\nonumber\\
\!\!\!&=&\!\!\!
\mathcal{K}_{j}\left(e^{\lambda|t|}|\xi|^{\gamma}\right)^{|\boldsymbol{l_j}||\boldsymbol{k_j}|+1}|r_1-r_2|.\nonumber\\
\end{eqnarray*}
Applying Lemma \ref{lm-GB5-1} (i) again, we have
\begin{eqnarray}\label{Delta-est}
\begin{split}
\Delta_{s}
\leq&\,
\sum_{j=1}^{s}\left(\frac{\mathcal{K}_{j}\left(e^{\lambda|t|}|\xi|^{\gamma}\right)^{|\boldsymbol{l_j}||\boldsymbol{k_j}|+1}|r_1-r_2|}
{(\boldsymbol{k_j}!)(\boldsymbol{l_j}!)^{|\boldsymbol{k_j}|}}\right)
\prod_{i=0}^{j-1}
\frac{\left(\varepsilon_{|\boldsymbol{l_i}|}S_{\boldsymbol{l_i}}
\left(e^{\lambda|t|}|\xi|^{\gamma}\right)^{|\boldsymbol{l_i}|}\right)^{|\boldsymbol{k_i}|}}
{(\boldsymbol{k_i}!)(\boldsymbol{l_i}!)^{|\boldsymbol{k_i}|}}\\
 &\, \times
\prod_{i=j+1}^{s+1}
\frac{\left(\varepsilon_{|\boldsymbol{l_{i}}|}S_{\boldsymbol{l_i}}
\left(e^{\lambda|t|}|\xi|^{\gamma}\right)^{|\boldsymbol{l_i}|}\right)^{|\boldsymbol{k_i}|}}
{(\boldsymbol{k_i}!)(\boldsymbol{l_i}!)^{|\boldsymbol{k_i}|}}=
\mathcal{J}_s\left(e^{\lambda|t|}|\xi|^{\gamma}\right)^{k+1}|r_1-r_2|.
\end{split}
\end{eqnarray}
Finally, combining (\ref{k-lip}), (\ref{Delta-est}) and Lemma \ref{lm-GB5-1}, we get that
\begin{eqnarray*}
\lefteqn{ |\frac{\partial^{k}}{\partial r^{k}} f(x(t,r_1,\xi),x(t-r_1,r_1,\xi))
-\frac{\partial^{k}}{\partial r^{k}} f(x(t,r_2,\xi),x(t-r_2,r_2,\xi))|}\\
\!\!\!&\leq&\!\!\!\!\!\!\!
\sum_{1\leq|\boldsymbol{\omega}|\leq k}
M_{|\boldsymbol{\omega}|+1}
\max\{|x(t,r_1,\xi)-x(t,r_2,\xi)|,|x(t-r_1,r_1,\xi)-x(t-r_2,r_2,\xi)|\}
\sum_{s=1}^{k}\sum_{p_s(\boldsymbol{\nu},\boldsymbol{\omega})}(k!)\\
\!\!\!&&\!\!\!\!\!\!\!
\times \prod_{j=1}^{s}
\frac{\left(\varepsilon_{|\boldsymbol{l_j}|}S_{\boldsymbol{l_j}}\left(e^{\lambda|t|}|\xi|^{\gamma}\right)^{|\boldsymbol{l_j}|}\right)^{|\boldsymbol{k_j}|}}
{(\boldsymbol{k_j}!)(\boldsymbol{l_j}!)^{|\boldsymbol{k_j}|}}
+\sum_{1\leq|\boldsymbol{\omega}|\leq k}M_{|\boldsymbol{\omega}|}
\sum_{s=1}^{k}\sum_{p_s(\boldsymbol{\nu},\boldsymbol{\omega})} (k!) \mathcal{J}_s\left(e^{\lambda|t|}|\xi|^{\gamma}\right)^{k+1}|r_1-r_2|\\
\!\!\!&\leq&\!\!\!\!
T_{(0,k+1)}
\left(e^{\lambda|t|}|\xi|^{\gamma}\right)^{k+1}|r_1-r_2|.
\end{eqnarray*}
Therefore, the proof is complete.
\end{proof}

\begin{lm}\label{lm-GB5-3}
Assume that {\bf (H2)} holds. Then
for any $x\in \mathcal{E}_{\gamma,\lambda}$,
and any $(t,r,\xi)$, $(t_1,r_1,\xi)$ and $(t_2,r_2,\xi)$ in $\mathbb{R}\times\Omega\times V_{\gamma}$
with $(t_1,r_1)\neq (t_2,r_2)$,
the following assertions holds:
\begin{itemize}
\item[(i)]
$|\mathcal{F}(x)(t,r,\xi)|
\leq \left(d+\varepsilon_0(Mx^{*}e^{x^{*}}+M_1\delta e^{x^{*}})/x^{*}+|f(0,0)|\delta /(ex^{*})\right)e^{\lambda|t|}|\xi|^{\gamma}.$

\item[(ii)]
$|\frac{\partial}{\partial t}\mathcal{F}(x)(t,r,\xi)|\leq
\left(Me^{x^{*}}\varepsilon_1+M_1e^{x^{*}}\varepsilon_0+|f(0,0)|\right)e^{\lambda|t|}|\xi|^{\gamma}$.

\item[(iii)]
For $\boldsymbol{\nu}=(\nu_1,\nu_2)$ with $\nu_1=0$ and $\nu_2=|\boldsymbol{\nu}|\leq k$,
we have
\begin{eqnarray*}
|\frac{\partial^{|\boldsymbol{\nu}|}}{\partial r^{|\boldsymbol{\nu}|}}\mathcal{F}(x)(t,r,\xi)|
\leq \left(MS_{(0,|\boldsymbol{\nu}|)}\varepsilon_{|\boldsymbol{\nu}|}+T_{(0,|\boldsymbol{\nu}|)}\delta /(|\boldsymbol{\nu}|x^{*})\right)
(e^{\lambda|t|}|\xi|^{\gamma})^{|\boldsymbol{\nu}|}.
\end{eqnarray*}

\item[(iv)]
For $\boldsymbol{\nu}=(\nu_1,\nu_2)$ with $1\leq\nu_i\leq |\boldsymbol{\nu}|-1$ and $2\leq|\boldsymbol{\nu}|\leq k$,
we have
\begin{eqnarray*}
|\frac{\partial^{|\boldsymbol{\nu}|}}{\partial t^{\nu_1}\partial r^{\nu_2}}\mathcal{F}(x)(t,r,\xi)|
\leq \left(MS_{(\nu_1,\nu_2)}\varepsilon_{|\boldsymbol{\nu}|}+T_{(\nu_1-1,\nu_2)}\right)
(e^{\lambda|t|}|\xi|^{\gamma})^{|\boldsymbol{\nu}|}.
\end{eqnarray*}

\item[(v)]
For $\boldsymbol{\nu}=(\nu_1,\nu_2)$ with $\nu_1=0$ and $\nu_2=k$,
we have
\begin{eqnarray*}
\lefteqn{|\frac{\partial^{k}}{\partial r^{k}}\mathcal{F}(x)(t_1,r_1,\xi)
-\frac{\partial^{k}}{\partial r^{k}}\mathcal{F}(x)(t_2,r_2,\xi)|}\\
\!\!\!&&\!\!\!\!\!\!\!\leq \left(MS_{(0,k+1)}\varepsilon_{k+1}+T_{(0,k+1)}\delta /((k+1)x^{*})+T_{(0,k)}\right)
\left(e^{\lambda|t_*|}|\xi|^{\gamma}\right)^{k+1}|(t_1,r_1)-(t_2,r_2)|.
\end{eqnarray*}
where $|t_*|=\max\{|t_1|,|t_2|\}$.
\item[(vi)]
For $\boldsymbol{\nu}=(\nu_1,\nu_2)$ with $1\leq \nu_1\leq k$ and $\nu_1+\nu_2=k$,
we have
\begin{eqnarray*}
\lefteqn{|\frac{\partial^{k}}{\partial t^{\nu_1}\partial r^{\nu_2}}\mathcal{F}(x)(t_1,r_1,\xi)
-\frac{\partial^{k}}{\partial t^{\nu_1}\partial r^{\nu_2}}\mathcal{F}(x)(t_2,r_2,\xi)|}\\
\!\!\!&&\!\!\!\!\!\!\!\leq
\left(MS_{(\nu_1,\nu_2+1)}\varepsilon_{k+1}+T_{(\nu_1,\nu_2)}+T_{(\nu_1-1,\nu_2+1)}\right)
\left(e^{\lambda|t_*|}|\xi|^{\gamma}\right)^{k+1}|(t_1,r_1)-(t_2,r_2)|.
\end{eqnarray*}
where $|t_*|=\max\{|t_1|,|t_2|\}$.
\end{itemize}
\end{lm}
\begin{proof}
For any $x\in \mathcal{E}_{\gamma,\lambda}$ and $(t,r,\xi) \in \mathbb{R}\times\Omega\times V_{\gamma}$, we have
\begin{equation}
\begin{split}
\!|\mathcal{F}(x)(t,r,\xi)| \leq&\,
|\xi|\!+\!|Ax(t-r,r,\xi)|\!+\!|\!\int_0^{t}\!|f(x(s,r,\xi),x(s-r,r,\xi))\!-\!f(0,0)|ds|\!
\\
&+\!|\!\int_0^{t}\!|f(0,0)|ds|\\
\leq&\,
|\xi|+M\varepsilon_{0}e^{\lambda(r+|t|)}|\xi|^{\gamma}
+|\int_0^{t}M_1 \varepsilon_{0}e^{\lambda(r+|s|)}|\xi|^{\gamma} ds|+|f(0,0)t|\\
\leq&\,
|\xi|+M\varepsilon_{0}e^{\lambda(r+|t|)}|\xi|^{\gamma}
+M_1 \varepsilon_{0}e^{\lambda(r+|t|)}|\xi|^{\gamma}/\lambda +|f(0,0)|e^{\lambda|t|}/(\lambda e)\label{Fx-est},
\end{split}
\end{equation}
where the last inequality follows from the fact that $\max_{t\in \mathbb{R}}|t|e^{-\lambda|t|}=1/(\lambda e)$.
As $\gamma=0$, by (\ref{Fx-est}) we obtain
\begin{eqnarray*}
|\mathcal{F}(x)(t,r,\xi)|e^{-\lambda|t|}
\leq d+(Me^{r\lambda}+M_1e^{r\lambda}/\lambda) \varepsilon_{0}+|f(0,0)|/(\lambda e),
\end{eqnarray*}
together with $0<r<\delta$ and $\delta \lambda=x^{*}$,
yields result (i) is true in the case $\gamma=0$.
As $\gamma=1$, by (\ref{Fx-est}) we have
\begin{eqnarray*}
|\mathcal{F}(x)(t,r,\xi)|(e^{\lambda|t|}|\xi|)^{-1}
\leq 1+(Me^{r\lambda}+M_1e^{r\lambda}/\lambda) \varepsilon_{0}+|f(0,0)|/(\lambda e|\xi|).
\end{eqnarray*}
Note that $0<r<\delta$, $\delta \lambda=x^{*}$ and $|\xi|>d>1$ for $\gamma=1$.
This implies that result {\bf (i)} is also true in the case $\gamma=1$.

Note that $x\in \mathcal{E}_{\gamma,\lambda}$ and $f$ satisfies {\bf (H2)}. We obtain
\begin{eqnarray*}
|\frac{\partial}{\partial t}\mathcal{F}(x)(t,r,\xi)|
\!\!\!&=&\!\!\!|A\frac{\partial}{\partial t}x(t-r,r,\xi)+f(x(t,r,\xi),x(t-r,r,\xi))|\\
\!\!\!&\leq&\!\!\!M\varepsilon_1e^{\lambda(r+|t|)}|\xi|^{\gamma}+|f(x(t,r,\xi),x(t-r,r,\xi))-f(0,0)|+|f(0,0)|\\
\!\!\!&\leq&\!\!\!M\varepsilon_1e^{\lambda(r+|t|)}|\xi|^{\gamma}+M_1\varepsilon_0e^{\lambda(r+|t|)}|\xi|^{\gamma}+|f(0,0)|\\
\!\!\!&\leq&\!\!\!\left(Me^{r\lambda}\varepsilon_1+M_1e^{r\lambda}\varepsilon_0+|f(0,0)|\right)e^{\lambda|t|}|\xi|^{\gamma},
\end{eqnarray*}
where the last inequality follows from the fact that $|\xi|>d>1$ in the case $\gamma=1$.
Then by the conditions $0<r<\delta $ and $\delta \lambda=x^{*}$, result (ii) is established.

For $\boldsymbol{\nu}=(\nu_1,\nu_2)$ with $\nu_1=0$ and $\nu_2=|\boldsymbol{\nu}|\leq k$,
using the Leibniz's Rule and Lemma \ref{lm-GB5-1}, we have
\begin{eqnarray*}
|\frac{\partial^{|\boldsymbol{\nu}|}}{\partial r^{|\boldsymbol{\nu}|}}\mathcal{F}(x)(t,r,\xi)|
\!\!\!&=&\!\!\!|A\frac{\partial^{|\boldsymbol{\nu}|}}{\partial r^{|\boldsymbol{\nu}|}}x(t-r,r,\xi)
+\int_{0}^{t}\frac{\partial^{|\boldsymbol{\nu}|}}{\partial r^{|\boldsymbol{\nu}|}}f(x(s,r,\xi),x(s-r,r,\xi))ds|\\
\!\!\!&\leq&\!\!\!
MS_{(0,|\boldsymbol{\nu}|)}\varepsilon_{|\boldsymbol{\nu}|}(e^{\lambda|t|}|\xi|^{\gamma})^{|\boldsymbol{\nu}|}
+|\int_{0}^{t}T_{(0,|\boldsymbol{\nu}|)}(e^{\lambda|s|}|\xi|^{\gamma})^{|\boldsymbol{\nu}|}ds|\\
\!\!\!&\leq&\!\!\!
\left(MS_{(0,|\boldsymbol{\nu}|)}\varepsilon_{|\boldsymbol{\nu}|}+T_{(0,|\boldsymbol{\nu}|)}/(|\boldsymbol{\nu}|\lambda)\right)
(e^{\lambda|t|}|\xi|^{\gamma})^{|\boldsymbol{\nu}|}.
\end{eqnarray*}
Together with $\delta \lambda=x^{*}$, result (iii) is true.

For $\boldsymbol{\nu}=(\nu_1,\nu_2)$ with $1\leq\nu_i\leq |\boldsymbol{\nu}|-1$ and $2\leq|\boldsymbol{\nu}|\leq k$,
using the Leibniz's Rule and Lemma \ref{lm-GB5-1} again, we have
\begin{eqnarray*}
|\frac{\partial^{|\boldsymbol{\nu}|}}{\partial t^{\nu_1}\partial r^{\nu_2}}\mathcal{F}(x)(t,r,\xi)|
\!\!\!&=&\!\!\!
|A\frac{\partial^{|\boldsymbol{\nu}|}}{\partial t^{\nu_1}\partial r^{\nu_2}}x(t-r,r,\xi)
+\frac{\partial^{|\boldsymbol{\nu}|-1}}{\partial t^{\nu_1-1}\partial r^{\nu_2}}f(x(t,r,\xi),x(t-r,r,\xi))|\\
\!\!\!&\leq&\!\!\!
MS_{(\nu_1,\nu_2)}\varepsilon_{|\boldsymbol{\nu}|}(e^{\lambda|t|}|\xi|^{\gamma})^{|\boldsymbol{\nu}|}+T_{(\nu_1-1,\nu_2)}
(e^{\lambda|t|}|\xi|^{\gamma})^{|\boldsymbol{\nu}|}.
\end{eqnarray*}
So result (iv) holds.

For $\boldsymbol{\nu}=(\nu_1,\nu_2)$ with $\nu_1=0$ and $\nu_2=k$,
applying Lemma \ref{lm-GB5-1} and Lemma \ref{lm-GB5-2}, we have
\begin{eqnarray*}
\lefteqn{|\frac{\partial^{k}}{\partial r^{k}}\mathcal{F}(x)(t_1,r_1,\xi)
-\frac{\partial^{k}}{\partial r^{k}}\mathcal{F}(x)(t_2,r_2,\xi)|}\nonumber\\
\!\!\!\!&\leq &\!\!\!\!
|A||\frac{\partial^{k}}{\partial r^{k}}x(t_1-r_1,r_1,\xi)-\frac{\partial^{k}}{\partial r^{k}}x(t_2-r_2,r_2,\xi)|
+|\int^{t_2}_{t_1} \frac{\partial^{k}}{\partial r^{k}} f(x(s,r_2,\xi),x(s-r_2,r_2,\xi))ds|\\
\!\!\!\!&&\!\!\!\! +|\int^{t_1}_{0} \left( \frac{\partial^{k}}{\partial r^{k}} f(x(s,r_1,\xi),x(s-r_1,r_1,\xi))
- \frac{\partial^{k}}{\partial r^{k}} f(x(s,r_2,\xi),x(s-r_2,r_2,\xi))\right)ds|\\
\!\!\!\!&\leq &\!\!\!\!
MS_{(0,k+1)}\varepsilon_{k+1}\left(e^{\lambda|t_*|}|\xi|^{\gamma}\right)^{k+1}\max\{|t_1-t_2|,|r_1-r_2|\}
+T_{(0,k)}\left(e^{\lambda|t_*|}|\xi|^{\gamma}\right)^{k+1}|t_1-t_2| \\
\!\!\!\!& &\!\!\!\!
+|\int^{t_1}_{0} T_{(0,k+1)}\left(e^{\lambda|s|}|\xi|^{\gamma}\right)^{k+1}|r_1-r_2|ds|\\
\!\!\!\!&\leq &\!\!\!\!
\left(MS_{(0,k+1)}\varepsilon_{k+1}+T_{(0,k+1)}/((k+1)\lambda)+T_{(0,k)}\right)
\left(e^{\lambda|t_*|}|\xi|^{\gamma}\right)^{k+1}\max\{|t_1-t_2|,|r_1-r_2|\},
\end{eqnarray*}
which implies that result (v) is true.

In the end, using Lemma \ref{lm-GB5-1},
we have for $\boldsymbol{\nu}=(\nu_1,\nu_2)$ with $1\leq \nu_1\leq k$ and $\nu_1+\nu_2=k$,
\begin{eqnarray*}
\lefteqn{|\frac{\partial^{k}}{\partial t^{\nu_1}\partial r^{\nu_2}}\mathcal{F}(x)(t_1,r_1,\xi)
-\frac{\partial^{k}}{\partial t^{\nu_1}\partial r^{\nu_2}}\mathcal{F}(x)(t_2,r_2,\xi)|}\\
\!\!\!\!&\leq &\!\!\!\!
|A||\frac{\partial^{k}}{\partial t^{\nu_1}\partial r^{\nu_2}}x(t_1-r_1,r_1,\xi)
-\frac{\partial^{k}}{\partial t^{\nu_1}\partial r^{\nu_2}}x(t_2-r_2,r_2,\xi)|\\
\!\!\!\!& &\!\!\!\!
+|\frac{\partial^{k-1}}{\partial t^{\nu_1-1}\partial r^{\nu_2}} f(x(t_1,r_1,\xi),x(t_1-r_1,r_1,\xi))
- \frac{\partial^{k-1}}{\partial t^{\nu_1-1}\partial r^{\nu_2}} f(x(t_2,r_1,\xi),x(t_2-r_1,r_1,\xi))|\\
\!\!\!\!& &\!\!\!\!
+|\frac{\partial^{k-1}}{\partial t^{\nu_1-1}\partial r^{\nu_2}} f(x(t_2,r_1,\xi),x(t_2-r_1,r_1,\xi))
- \frac{\partial^{k-1}}{\partial t^{\nu_1-1}\partial r^{\nu_2}} f(x(t_2,r_2,\xi),x(t_2-r_2,r_2,\xi))|\\
\!\!\!\!&\leq&\!\!\!\!
\left(MS_{(\nu_1,\nu_2+1)}\varepsilon_{k+1}+T_{(\nu_1,\nu_2)}+T_{(\nu_1-1,\nu_2+1)}\right)
\left(e^{\lambda|t_*|}|\xi|^{\gamma}\right)^{k+1}\max\{|t_1-t_2|,|r_1-r_2|\}.
\end{eqnarray*}
Therefore, result (vi) holds. Then the proof of Lemma \ref{lm-GB5-3} is complete.
\end{proof}

For any constant $d_0>0$, we define
$$\mathcal{E}^{\lambda}_{d_0}:=\{x\in C(\mathbb{R}\times\Omega\times V^{0}_{d_0}, \mathbb{R}^n):
\sup_{(t,r,\xi)\in\mathbb{R}\times\Omega\times V^{0}_{d_0}} |x(t,r,\xi)|e^{-\lambda|t|}<+\infty\}$$
equipped with the norm
$$|x|_{\mathcal{E}^{\lambda}_{d_0}}=
\sup_{(t,r,\xi)\in\mathbb{R}\times\Omega\times V^{0}_{d_0}} |x(t,r,\xi)|e^{-\lambda|t|}.$$
Similarly to Lemma \ref{lm-4-complete}, we can prove $(\mathcal{E}^{\lambda}_{d_0}, |\cdot|_{\mathcal{E}^{\lambda}_{d_0}})$ is a Banach space.
\begin{lm}\label{lm-GB5-4}
Assume that {\bf (H2)} holds and $0<r<\delta\leq r_0$.
Then $\mathcal{F}$ in (\ref{map-2}) has a unique fixed point in $\mathcal{E}^{\lambda}_{d_0}$.
\end{lm}

\begin{proof}
Obviously, for each $x\in \mathcal{E}^{\lambda}_{d_0}$, $\mathcal{F}(x)$ is a continuous map from $\mathbb{R}\times\Omega\times V^{0}_{d_0}$ to $\mathbb{R}^n$.
Using the same procedure as for Lemma \ref{lm-GB5-3} (i), we have
\begin{eqnarray*}
|\mathcal{F}(x)(t,r,\xi)|
\leq \left(d_0+|x|_{\mathcal{E}^{\lambda}_{d_0}}(Mx^{*}e^{x^{*}}+M_1r_0 e^{x^{*}})/x^{*}+|f(0,0)|r_0 /(ex^{*})\right)e^{\lambda|t|}.
\end{eqnarray*}
Then $\mathcal{F}$ maps $\mathcal{E}^{\lambda}_{d_0}$ into $\mathcal{E}^{\lambda}_{d_0}$. Moreover, for
any $x,y\in\mathcal{E}^{\lambda}_{d_0}$, we observe that
\begin{eqnarray*}
\lefteqn{|\mathcal{F}(x)(t,r,\xi)-\mathcal{F}(y)(t,r,\xi)|}\nonumber\\
\!\!\!&\leq&\!\!\!\!\!
|A||x(t-r,r,\xi)-y(t-r,r,\xi)|+|\!\! \int_0^{t} \!\!|f(x(s,r,\xi),x(s-r,r,\xi))-f(y(s,r,\xi),y(s-r,r,\xi))|ds|\nonumber\\
\!\!\!&\leq&\!\!\!\!\!
M|x-y|_{\mathcal{E}^{\lambda}_{d_0}}e^{\lambda (r+|t|)}+
|\int_0^{t} M_1 |x-y|_{\mathcal{E}^{\lambda}_{d_0}} e^{\lambda (r+|s|)}  ds|\nonumber\\
\!\!\!&\leq&\!\!\!\!\!
\left(Me^{r\lambda}+M_1e^{r\lambda}/\lambda\right)e^{\lambda |t|}|x-y|_{\mathcal{E}^{\lambda}_{d_0}},
\label{T-contr-lm6-5}
\end{eqnarray*}
which implies that
$|\mathcal{F}(x)-\mathcal{F}(y)|_{\mathcal{E}^{\lambda}_{d_0}}\leq (Me^{x^{*}}+M_1r_0 e^{x^{*}}/x^{*})|x-y|_{\mathcal{E}^{\lambda}_{d_0}}$.
Using (\ref{x-star-1}) yields that $\mathcal{F}$ is a contraction. By the Contraction
Mapping Principle, $\mathcal{F}$ has a unique fixed point in $\mathcal{E}^{\lambda}_{d_0}$. Then
the proof of Lemma \ref{lm-GB5-4} is complete.
\end{proof}

Next we prove Theorem \ref{thm-2}.
\begin{proof}[Proof of Theorem \ref{thm-2}]
We prove this theorem in three steps.
\vskip 0.2cm
\noindent{\bf Step (i).}
We first prove that there exists a continuous map $\Psi_1$ satisfying Theorem \ref{thm-2} (i) and (iii).
Let the constant $r_0$ satisfy Hypothesis {\bf (H2)}.
By (\ref{x-star-1}), we see $(Mx^{*}e^{x^{*}}+M_1r_0e^{x^{*}})/x^{*}<1$. Then we can choose a sufficiently large $\varepsilon_0>0$
such that
\begin{eqnarray}\label{esp-0-est}
(1-(Mx^{*}e^{x^{*}}+M_1r_0 e^{x^{*}})/x^{*})\varepsilon_0\geq d+|f(0,0)|r_0 /(ex^{*}).
\end{eqnarray}
By (\ref{x-star-2}), we note that for $\boldsymbol{\nu}=(\nu_1,\nu_2)$ with $1\leq|\boldsymbol{\nu}|\leq k$,
\begin{eqnarray*}
1-Me^{x^{*}}-MS_{(0,1)}>0,\ \ \ \ 1-M\sum_{\nu_1=0}^{|\boldsymbol{\nu}|}\frac{|\boldsymbol{\nu}|!S_{\boldsymbol{\nu}}}{(\nu_1!) (\nu_2!)}>0,
\end{eqnarray*}
and for $\boldsymbol{\nu}=(\nu_1,\nu_2)$ with $|\boldsymbol{\nu}|=k$,
\begin{eqnarray*}
1-M\sum_{\nu_1=0}^{k}\frac{k!S_{(\nu_1,\nu_2+1)}}{(\nu_1!) (\nu_2!)}>0.
\end{eqnarray*}
%
Then we can choose the constants $\varepsilon_{j}$ and $\delta_{j}$ for $j=1,...,k+1$ in the following way:
\begin{equation}\label{e-1-d-1}
\begin{split}
(1-Me^{x^{*}}-MS_{(0,1)})\varepsilon_1/2 &\geq M_1e^{x^{*}}\varepsilon_0+|f(0,0)|,
\\
\delta_1 &:=(1-Me^{x^{*}}-MS_{(0,1)})\varepsilon_1x^{*}/(2T_{(0,1)}),
\end{split}
\end{equation}
\begin{equation}\label{e-v-d-v}
\begin{split}
\left(1-M\sum_{\nu_1=0}^{|\boldsymbol{\nu}|}\frac{|\boldsymbol{\nu}|!S_{\boldsymbol{\nu}}}{(\nu_1!) (\nu_2!)}\right)
\varepsilon_{|\boldsymbol{\nu}|}
& \geq 2 \sum_{\nu_1=1}^{|\boldsymbol{\nu}|}\frac{|\boldsymbol{\nu}|!T_{(\nu_1-1,\nu_2)}}{(\nu_1!) (\nu_2!)},
\\
\delta_{|\boldsymbol{\nu}|}
& :=\frac{\left(1-M\sum_{\nu_1=0}^{|\boldsymbol{\nu}|}\frac{|\boldsymbol{\nu}|!S_{\boldsymbol{\nu}}}{(\nu_1!) (\nu_2!)}\right)
\varepsilon_{|\boldsymbol{\nu}|} |\boldsymbol{\nu}| x^{*}}
{2T_{(0,|\boldsymbol{\nu}|)}},\, 2\leq |\boldsymbol{\nu}|\leq k,
\end{split}
\end{equation}
\begin{equation}\label{d-k+1}
\begin{split}
\left(1-M\sum_{\nu_1=0}^{k}\frac{k!S_{(\nu_1,\nu_2+1)}}{(\nu_1!) (\nu_2!)}\right)\varepsilon_{k+1}
&\geq
2 \sum_{\nu_1=1}^{k}\frac{k!\left(T_{(\nu_1,\nu_2)}+T_{(\nu_1-1,\nu_2+1)}\right)}{(\nu_1!) (\nu_2!)}+2T_{(0,k)},\\
\delta_{|\boldsymbol{\nu}|}&:=
\frac{\left(1-M\sum_{\nu_1=0}^{k}\frac{k!S_{(\nu_1,\nu_2+1)}}{(\nu_1!) (\nu_2!)}\right)\varepsilon_{k+1}(k+1) x^{*}}{2T_{(0,k+1)}}.
\end{split}
\end{equation}
By Lemma \ref{lm-GB5-1}, we further observe that for each $\boldsymbol{\nu}$,
$T_{\boldsymbol{\nu}}$ only depends on $\varepsilon_0$,...,$\varepsilon_{|\boldsymbol{\nu}|}$,
which guarantees that all $\varepsilon_{j}$ and $\delta_{j}$ are well defined.
Take $\delta:=\min\{r_0, \delta_1, \delta_2, \ldots, \delta_{k+1}\}$.
Clearly, $\delta$ is independent of $\mathbb{\xi}\in \mathbb{R}^{n}$.
Applying the above parameters, for each $\gamma \in \{0,1\}$ we define the set $\mathcal{E}_{\gamma,\lambda}$.
According to Lemma \ref{lm-5-complete}, for each $\gamma \in \{0,1\}$,
$(\mathcal{E}_{\gamma,\lambda}, \rho_{\gamma})$ is a complete metric space.

Obviously, for each $x\in \mathcal{E}_{\gamma,\lambda}$, $\mathcal{F}(x)$ is a continuous map from
$\mathbb{R}\times\Omega\times V_{\gamma}$ to $\mathbb{R}^n$ and
$\mathcal{F}(x)(t,r, \xi)$ is $C^k$ in $(t,r)$ for each $\xi\in V_{\gamma}$.
By (\ref{esp-0-est}) and Lemma \ref{lm-GB5-3} (i),
we find that $|\mathcal{F}(x)|_{\mathcal{E}_{\gamma,\lambda}}\leq \varepsilon_0$.
By (\ref{e-1-d-1})-(\ref{d-k+1}), Lemma \ref{derivt} and Lemma \ref{lm-GB5-3},
for $0<r<\delta$ we find that $|\mathcal{F}(x)|_{\mathcal{E}_{\gamma,\lambda},j}\leq \varepsilon_j$ for $j=1,...,k+1$.
Then $\mathcal{F}$ maps $\mathcal{E}_{\gamma,\lambda}$ into itself.
Moreover, similarly to Lemma \ref{lm-GB5-4}, we also prove that $\mathcal{F}$ is a contraction in $\mathcal{E}_\gamma$.
By the Contraction Mapping Principle, $\mathcal{F}$ has a unique fixed point in the complete metric space $(\mathcal{E}_{\gamma,\lambda}, \rho_{\gamma})$.
We denote this fixed point by $\Psi_1$.
Then $\Psi_1$ satisfies Theorem \ref{thm-2} (i) and (iii).
\vskip 0.2cm
\noindent{\bf Step (ii).} Secondly, we prove that there exists a continuous map $\Psi_2$ satisfying Theorem \ref{thm-2} (ii).
For any constant $d_0>0$, by the similar method of Theorem \ref{thm-1},
there exists a constant $\delta$ which is independent of $d_0$
such that the operator $\mathcal{F}$ has a unique fixed point $\Psi_2$, which is
a continuous map from $\mathbb{R}\times\Omega\times V^{0}_{d_0}$ to $\mathbb{R}^n$, and
for each fixed $(t,r)\in\mathbb{R}\times\Omega$, $\Psi_2(t,r,\cdot\,)$ is $C^{k,1}$
and there exists a sequence $\{\beta_{j}\}_{j=0}^{k+1}$ such that
\begin{eqnarray*}
\sup_{(t,r,\xi)\in\mathbb{R}\times\Omega\times V^{0}_{d_0}} |\Psi_2(t,r,\xi)|e^{-\lambda|t|} \!\!\! &\leq&\!\!\! \beta_0,\\
\sup_{(t,r,\xi)\in\mathbb{R}\times\Omega\times V^{0}_{d_0}} |D_2^j\Psi_2(t,r,\xi)|e^{-j\lambda|t|} \!\!\!&\leq&\!\!\! \beta_j, \ \  j=1,2,...,k,\nonumber\\
\sup_{\xi_1\neq\xi_2}\frac{|D_2^k \Psi_2(t,r,\xi_1)-D_2^k \Psi_2(t,r,\xi_2)|}{|\xi_1-\xi_2|}e^{-(k+1)\lambda|t|} \!\!\!&\leq&\!\!\! \beta_{k+1},
\end{eqnarray*}
where the constant $\lambda$ satisfies the conditions stated in {\bf (H2)}.
Without loss of generality, we assume that $\delta$ in this step is the same as the one in step 1,
otherwise, we choose the minimum one.
\vskip 0.2cm
\noindent{\bf Step (iii).} Finally, we prove $\Psi_1=\Psi_2$ on $\mathbb{R}\times\Omega\times V^{0}_{d_0}$.
By step 1 we see that $\Psi_1\in \mathcal{E}_{\gamma,\lambda}$.
Then by the property of $\mathcal{E}_{\gamma,\lambda}$, we have $\Psi_1\in \mathcal{E}^{\lambda}_{d_0}$.
By step 2, we also have $\Psi_2\in \mathcal{E}^{\lambda}_{d_0}$. In view of Lemma \ref{lm-GB5-4}, we find
$\Psi_1=\Psi_2$ on $\mathbb{R}\times\Omega\times V^{0}_{d_0}$.
Note that the arbitrariness of $d_0$ and the fact that $\delta$ is independent of $d_0$.
We have $\Psi_1=\Psi_2$ on $\mathbb{R}\times\Omega\times \mathbb{R}^{n}$.
Therefore, the proof of Theorem \ref{thm-2} is complete.
\end{proof}
%

\section{Illustrating example}
In this section we apply the main results obtained in Theorem \ref{thm-2} to
study the dynamics of the following neutral differential system
\begin{eqnarray}
\label{VDP-1}
\begin{split}
\frac{d}{dt}\left\{\tilde{x}_{1}(t)-c\tilde{x}_{1}(t-\tilde{r})\right\} &= \tilde{x}_{2}(t)-\tilde{\varepsilon}\left\{\frac{1}{2}\tilde{x}_{1}^{2}(t)+\frac{1}{3}\tilde{x}_{1}^{3}(t)\right\},
\\
\frac{d}{dt} \tilde{x}_{2}(t)  &= b-\tilde{x}_{1}(t-\tilde{r}),
\end{split}
\end{eqnarray}
where $(\tilde{x}_{1},\tilde{x}_{2})\in\mathbb{R}^{2}$, $b\in\mathbb{R}$, $1>c>0$, $\tilde{\varepsilon}>0$ and $\tilde{r}>0$.
Letting $c=0$ and $\tilde{r}=0$, system (\ref{VDP-1})
becomes the well-known van der Pol (abbreviated as vdP) oscillator model.
Then we call system (\ref{VDP-1}) a vdP oscillator model of neutral type.
Here we are interested in the case  $1/\tilde{r}\gg \tilde{\varepsilon}^{3}\gg 1$,
that is,  the vdP oscillator model (\ref{VDP-1})  is in the relaxation case and has a small time delay.
By a rescaling
$(\tilde{x}_{1},\tilde{x}_{2},t)\to (\tilde{x}_{1},\tilde{\varepsilon}\tilde{x}_{2},t/\tilde{\varepsilon})$,
the vdP oscillator model (\ref{VDP-1})  is changed into
\begin{eqnarray}
\label{VDP-2}
\begin{split}
\frac{d}{dt}\left\{x_{1}(t)-cx_{1}(t-r)\right\} &= x_{2}(t)-\left\{\frac{1}{2}x_{1}^{2}(t)+\frac{1}{3}x_{1}^{3}(t)\right\}:=f_{1}(x_{1}(t),x_{2}(t)),
\\
\frac{d}{dt} x_{2}(t)  &= \varepsilon(b- x_{1}(t-r)):=f_{2}(x_{1}(t-r)),
\end{split}
\end{eqnarray}
where $x_{1}(t)=\tilde{x}_{1}(t/\tilde{\varepsilon})$, $x_{2}(t)=\tilde{x}_{2}(t/\tilde{\varepsilon})$,
$\varepsilon=1/\tilde{\varepsilon}^2$, $r=\tilde{\varepsilon}\tilde{r}$ and $0<r\ll \varepsilon\ll 1$.
When $r=0$,
then system (\ref{VDP-2}) is changed into a standard slow-fast system
\begin{eqnarray}
\label{VDP-3}
\begin{split}
(1-c)\frac{d}{dt}x_{1}(t) &= x_{2}(t)-\left\{\frac{1}{2}x_{1}^{2}(t)+\frac{1}{3}x_{1}^{3}(t)\right\},
\\
\frac{d}{dt} x_{2}(t)  &= \varepsilon(b- x_{1}(t)).
\end{split}
\end{eqnarray}
In distinct parameter regions,
complex oscillation phenomena occur near the following curve
\begin{eqnarray*}
\mathcal{W}_{1}:=\left\{(x_{1},x_{2})\in\mathbb{R}^{2}: x_{2}=\frac{1}{2}x_{1}^{2}+\frac{1}{3}x_{1}^{3}, \ -3/2\leq x_{1}\leq 1/2\right\},
\end{eqnarray*}
which contains all non-hyperbolic points of the critical manifold for the slow-fast system (\ref{VDP-3}).
In particular, under some suitable conditions the slow-fast system (\ref{VDP-3})
has hyperbolic periodic orbits, such as relaxation oscillations and canard cycles
(see, for instance, \cite{Dumortieretal-Roussarie-96}).
One of interesting problems is to study the effect of small delay $r$ on these periodic orbits
near the curve $\mathcal{W}_{1}$.
Without loss of generality,
we consider the restriction of system (\ref{VDP-2}) on the set
$$\mathcal{W}_{2}:=\left\{(x_{1},x_{2})\in\mathbb{R}^{2}: -2\leq x_{1}\leq 1, \ -1/2\leq x_{2}\leq 1/2\right\},$$
the interior of which includes the set $\mathcal{W}_{1}$.
Note that both
$f_{1}$ and $f_{2}$ are polynomials,
then for a large positive integer $k$,
there exists  constants $M_{j}$, $j=1,...,k+1$,
such that the restriction of $f=(f_{1},f_{2})^{T}$ on the set $\mathcal{W}_{2}$ satisfies {\bf (H2)}.
By applying the cut-off technique,
we can get a modified system which is consistent with the original system (\ref{VDP-2}) in the set $\mathcal{W}_{2}$
and satisfies {\bf (H2)}.
More specifically,
let the cut-off function $\chi:[0,+\infty)\to [0,1]$ satisfy the following properties:
{\rm (i')} $\chi\in C^{\infty}$;
{\rm (ii')} $\chi(u)=1$ for $u\in[0,1]$ and $\chi(u)=0$ for $u\in[2,+\infty)$;
{\rm (iii')} $\sup_{u \in [0,+\infty)} |\chi'(u)| \le 2$.
We define the modified function $\tilde{f}$ of $f$ by
$\tilde{f}(\phi(0),\phi(-r))=f(\chi(|\phi|/\kappa)\phi(0),\chi(|\phi|/\kappa)\phi(-r))$
for each $\phi\in\mathcal{C}$ and a certain positive constant $\kappa$.
Then by Theorem \ref{thm-2},
the modified system has a two-dimensional inertial manifold which is $C^{k,1}$ in $r$.
Consider the restriction of the modified system on this inertial manifold
and let $r=0$ in  system (\ref{VDP-2}).
We obtain that  the zeroth-order approximation of the restricted system  in the set $\mathcal{W}_{2}$ has the form (\ref{VDP-3}).
The high-order approximations of the restricted system can be obtained by the similar method used in
\cite{Chicone03,Chicone04} for retarded differential equations.
Note that the restricted system can be given by a two-dimensional ordinary differential system,
then by the structural stability of hyperbolic periodic orbits,
we can obtain that system (\ref{VDP-2})  with small delay $r$ also have periodic orbits
near the hyperbolic relaxation oscillations  and hyperbolic canard cycles arising in the slow-fast system (\ref{VDP-3}).
As a result, Theorem \ref{thm-2} is useful to study the dynamics of the vdP oscillator model (\ref{VDP-1}) with small delay.

\section{Discussion}

We hope that the method used here could be applied to
establish the existence of smooth invariant manifolds,
such as stable manifolds, unstable manifolds, center manifolds and so on, for other classes of evolutionary equations.
In particular, to the best of our knowledge,
the existence and the smoothness  of inertial manifolds for neutral differential equations with arbitrary delays
are not understood yet.
We also point out that our results are not sharp.
It is interesting to give a sharp smallness condition
such that equations (\ref{NA-NDE}) and (\ref{A-NDE}) both have $C^{k,1}$ inertial manifolds.

It is also possible to make a weaker condition on $M$ to obtain smooth inertial manifolds
for equations (\ref{NA-NDE}) and (\ref{A-NDE}).
All these conditions on  $M$ in {\bf (H1)} and  {\bf (H2)} are used to guarantee that the Contraction Mapping Principle is valid.
It should be noted especially that,
to guarantee the smoothness of the inertial manifold for equation (\ref{A-NDE}) with respect to the delay $r$,
we demand that the contraction condition $M$ in {\bf (H2)} is sharper than one in {\bf (H1)}.
We hope that all these conditions on  $M$ can be relaxed.

It is also interesting  to study whether there exists an inertial manifold of equation  (2.3) with some smallness conditions
is analytic (resp. $C^{\infty}$) in the delay $r$ if the function $f$ is analytic (resp. $C^{\infty}$).
However, for retarded differential equation (2.3) (with $A=0$),
our method could prove that equation (2.3) admits an $C^{\infty}$ inertial manifold in the delay $r$
if the function $f$ is $C^{\infty}$.
Complex oscillations arising from the vdP oscillator model (\ref{VDP-2}) and the difference between the original system
and its approximate systems should also be further studied.

\subsection*{Appendix A}\label{sec-app}
 \renewcommand\theequation{A.1}
Let $H(x):=x e^{-x}-Mx$, $x\in \mathbb{R}$. We consider the following equation
\begin{eqnarray}\label{exp-poly-1}
H(x)=M_1r, \ \ \ \  x\in \mathbb{R},
\end{eqnarray}
where the constants $M>0$, $M_1>0$ and $r>0$ are defined in Hypothesis {\bf (H1)} (resp. {\bf (H2)}).
Then equation (\ref{exp-poly-1}) has positive real roots if and only if
$0<M<1$ and $0<M_1r\leq H(x_0)$,
where $x_0\in(0,1)$ is the unique positive real zero of $H'$,
the derivative of $H$ with respect to $x$.
If $0<M<1$ and $0<M_1r< H(x_0)$,
then there exist exactly two positive real roots of equation (\ref{exp-poly-1}),
which are denoted by $x_1(r)$ and $x_2(r)$
with $0<x_1(r)<x_0<x_2(r)<-\ln M$,
and $H(x)>M_1r$ for $x\in (x_1(r),x_2(r))$.
\begin{proof}
(Necessity)
Suppose that equation (\ref{exp-poly-1}) has positive real roots,
we first prove $0<M<1$. Otherwise,
suppose that the constant $M$ satisfies $M\geq 1$.
Clearly, the first and second order derivative of $H$ with respect to $x$ are in the form
$H'(x)=(1-x)e^{-x}-M$ and $H^{''}(x)=(x-2)e^{-x}$ for $x\in \mathbb{R}$, respectively.
Then, we see that
$H'(0)=1-M>-M$, $H'(x)\leq -M$ for $x\geq 1$ and $H'$ decreases monotonely in the interval $[0,1]$.
Thus, $\max_{x\geq0}H'(x)=H'(0)=1-M\leq 0$,
which implies  $H(x)\leq H(0)=0< M_1r$ for $x\geq 0$.
Hence, equation (\ref{exp-poly-1}) has no positive real roots,
which is a contraction.
Thus, the constant $M$ satisfies $0<M<1$.

Note that for $0<M<1$, $H'(0)=1-M>0$, $H'(x)\leq -M<0$ for $x\geq 1$,
and $H'$ decreases monotonely in the interval $[0,1]$.
Then there is exactly one zero of $H^{'}$ on $[0,+\infty)$,
which is denoted by $x_0$ with $0<x_0<1$,
further, $H'(x)\geq 0$ for $x\in [0, x_0]$
and $H'(x)\leq 0$ for $x\in [x_0, +\infty)$.
Hence, $\max_{x\geq0}H(x)=H(x_0)>0$.
Thus, if $H(x)=M_1 r$ has positive real roots,
then it is necessary that $M_1 r\leq H(x_0)$.
Therefore, the necessity is proved.

(Sufficiency)
If $0<M<1$, from the proof of the necessity, we see that $x_0$ is well defined.
Note that for $0<M<1$, $H$ has a positive real zero $-\ln M$,
thus the sufficiency is proved.

Assume $0<M<1$, then by the statements above, we see that
$H$ increases monotonely in $[0,x_0]$ and decreases monotonely in $[x_0,+\infty)$,
$\max_{x\geq0} H(x)=H(x_0)>H(0)=0$,
and $H$ has exactly two nonnegative real zeros, that is, $0, -\ln M$.
Hence, for $0<M_1r< H(x_0)$,
equation (\ref{exp-poly-1}) has exactly two positive real zeros
$x_1(r)$ and $x_2(r)$  with $0<x_1(r)<x_0<x_2(r)<-\ln M$,
and $H(x)> M_1r$ for $x\in (x_1(r),x_2(r))$.
Therefore, the proof is now complete.
\end{proof}

\subsection*{Appendix B. Proof of (\ref{est-G-dev})}
By Lemma \ref{partial-devt}, we obtain
\renewcommand\theequation{B.1}
\begin{equation}
\begin{split}
 \label{prop-2-1}
\lefteqn{D^k_2G(t,\xi_1)-D^k_2G(t,\xi_2)}\\
=&\,  k!\sum_{m=1}^{k}D^m_2F(t,y_t)\sum_{p(k,m)}\prod_{i=1}^{k}\frac{(D^{i}_2y_t)^{\omega_i}}{(\omega_i!)(i!)^{\omega_i}}
-k!\sum_{m=1}^{k}D^m_2F(t,z_t)\sum_{p(k,m)}\prod_{i=1}^{k}\frac{(D^{i}_2z_t)^{\omega_i}}{(\omega_i!)(i!)^{\omega_i}}\\
=&\, k!\sum_{m=1}^{k}(D^m_2F(t,y_t)-D^m_2F(t,z_t))\sum_{p(k,m)}\prod_{i=1}^{k}\frac{(D^{i}_2y_t)^{\omega_i}}{(\omega_i!)(i!)^{\omega_i}} +k!\sum_{m=1}^{k}D^m_2F(t,z_t)\sum_{p(k,m)}\widetilde{I}_{2}^{k},
\end{split}
\end{equation}
where
\renewcommand\theequation{B.2}
\begin{equation}
\begin{split}
\label{prop-2-2}
\widetilde{I}_{2}^{k}
:=&\,
\prod_{i=1}^{k} \frac{(D^{i}_2y_t)^{\omega_i}}{(\omega_i!)(i!)^{\omega_i}}
-\prod_{i=1}^{k} \frac{(D^{i}_2z_t)^{\omega_i}}{(\omega_i!)(i!)^{\omega_i}}\\
=&\,
\frac{\widetilde{Q}_{1}}{(\omega_1!)(1!)^{\omega_1}}\prod_{i=2}^{k} \frac{(D^{i}_2y_t)^{\omega_i}}{(\omega_i!)(i!)^{\omega_i}}
+\frac{(D^{1}_2z_t)^{\omega_1}}{(\omega_1!)(1!)^{\omega_1}}\frac{\widetilde{Q}_{2}}{(\omega_2!)(2!)^{\omega_2}}
\prod_{i=3}^{k} \frac{(D^{i}_2y_t)^{\omega_i}}{(\omega_i!)(i!)^{\omega_i}}\\
& +\cdot\cdot\cdot+\prod_{i=1}^{k-1} \frac{(D^{i}_2z_t)^{\omega_i}}{(\omega_i!)(i!)^{\omega_i}}\frac{\widetilde{Q}_{k}}{(\omega_k!)(k!)^{\omega_k}}\\
=&\,
\sum_{j=1}^{k}
\left(\prod_{i=0}^{j-1} \frac{(D^{i}_2z_t)^{\omega_i}}{(\omega_i!)(i!)^{\omega_i}}\right)
\left(\frac{\widetilde{Q}_{j}}{(\omega_j!)(j!)^{\omega_j}}\right)
\left(\prod_{i=j+1}^{k+1} \frac{(D^{i}_2y_t)^{\omega_i}}{(\omega_i!)(i!)^{\omega_i}}\right),\\
\end{split}
\end{equation}
\renewcommand\theequation{B.3}
\begin{equation}
\begin{split}
\label{prop-2-3}
\widetilde{Q}_{j}
:=&\, (D^{j}_2y_t)^{\omega_j}-(D^{j}_2z_t)^{\omega_j}\\
=&\,
(D^{j}_2y_t-D^{j}_2z_t)(D^{j}_2y_t)^{\omega_j-1}+D^{j}_2z_t(D^{j}_2y_t-D^{j}_2z_t)(D^{j}_2y_t)^{\omega_j-2}\\
&\,+\cdot\cdot\cdot+(D^{j}_2z_t)^{\omega_j-1}(D^{j}_2y_t-D^{j}_2z_t).
\end{split}
\end{equation}
Along with (\ref{prop-2-1})-(\ref{prop-2-3}),
it is easy to verify that (\ref{est-G-dev}) holds.

\subsection*{Appendix C. Proof of (\ref{k-lip})}
By Lemma \ref{partial-devt}, we have
\renewcommand\theequation{C.1}
\begin{equation}\label{prop-3-1}
\begin{split}
& \lefteqn{ \frac{\partial^{k}}{\partial r^{k}} f(x(t,r_1,\xi),x(t-r_1,r_1,\xi))
-\frac{\partial^{k}}{\partial r^{k}} f(x(t,r_2,\xi),x(t-r_2,r_2,\xi))}\\
=&\,
\sum_{1\leq|\boldsymbol{\omega}|\leq k}
f_{\boldsymbol{\omega}}(r_1)
\sum_{s=1}^{k}\!\sum_{p_s(\boldsymbol{\nu},\boldsymbol{\omega})}
(k!)\prod_{j=1}^{s}
\frac{({\bf g}_{\boldsymbol{l_j}}(r_1))^{\boldsymbol{k_j}}}{(\boldsymbol{k_j}!)(\boldsymbol{l_j}!)^{|\boldsymbol{k_j}|}}\\
&\,
-\sum_{1\leq|\boldsymbol{\omega}|\leq k}
f_{\boldsymbol{\omega}}(r_2)
\sum_{s=1}^{k}\!\sum_{p_s(\boldsymbol{\nu},\boldsymbol{\omega})}
(k!)\prod_{j=1}^{s}
\frac{({\bf g}_{\boldsymbol{l_j}}(r_2))^{\boldsymbol{k_j}}}{(\boldsymbol{k_j}!)(\boldsymbol{l_j}!)^{|\boldsymbol{k_j}|}}\\
=&\,
\sum_{1\leq|\boldsymbol{\omega}|\leq k}\!\!
(f_{\boldsymbol{\omega}}(r_1)-f_{\boldsymbol{\omega}}(r_2))
\sum_{s=1}^{k}\!\sum_{p_s(\boldsymbol{\nu},\boldsymbol{\omega})}
(k!)\prod_{j=1}^{s}
\frac{({\bf g}_{\boldsymbol{l_j}}(r_1))^{\boldsymbol{k_j}}}{(\boldsymbol{k_j}!)(\boldsymbol{l_j}!)^{|\boldsymbol{k_j}|}}\\
&\,+\sum_{1\leq|\boldsymbol{\omega}|\leq k}\!\! f_{\boldsymbol{\omega}}(r_2)
\sum_{s=1}^{k}\!\sum_{p_s(\boldsymbol{\nu},\boldsymbol{\omega})}(k!) \widetilde{\Delta}_{s},
\end{split}
\end{equation}
where
\renewcommand\theequation{C.2}
\begin{equation}
\begin{split}
\label{prop-3-2}
\widetilde{\Delta}_{s}:
=&\,
\prod_{j=1}^{s}\frac{({\bf g}_{\boldsymbol{l_j}}(r_1))^{\boldsymbol{k_j}}}{(\boldsymbol{k_j}!)(\boldsymbol{l_j}!)^{|\boldsymbol{k_j}|}}
-\prod_{j=1}^{s}\frac{({\bf g}_{\boldsymbol{l_j}}(r_2))^{\boldsymbol{k_j}}}{(\boldsymbol{k_j}!)(\boldsymbol{l_j}!)^{|\boldsymbol{k_j}|}}\\
=&\,
\frac{\widetilde{\Theta}_{1}}{(\boldsymbol{k_1}!)(\boldsymbol{l_1}!)^{|\boldsymbol{k_1}|}}
\prod_{j=2}^{s}\frac{({\bf g}_{\boldsymbol{l_j}}(r_1))^{\boldsymbol{k_j}}}{(\boldsymbol{k_j}!)(\boldsymbol{l_j}!)^{|\boldsymbol{k_j}|}}
+\frac{({\bf g}_{\boldsymbol{l_1}}(r_2))^{\boldsymbol{k_1}}}{(\boldsymbol{k_1}!)(\boldsymbol{l_1}!)^{|\boldsymbol{k_1}|}}
\frac{\widetilde{\Theta}_{2}}{(\boldsymbol{k_2}!)(\boldsymbol{l_2}!)^{|\boldsymbol{k_2}|}}
\prod_{j=3}^{s}\frac{({\bf g}_{\boldsymbol{l_j}}(r_1))^{\boldsymbol{k_j}}}{(\boldsymbol{k_j}!)(\boldsymbol{l_j}!)^{|\boldsymbol{k_j}|}} \\
&\, +\cdot\cdot\cdot+
\prod_{j=1}^{s-1}\frac{({\bf g}_{\boldsymbol{l_j}}(r_2))^{\boldsymbol{k_j}}}{(\boldsymbol{k_j}!)(\boldsymbol{l_j}!)^{|\boldsymbol{k_j}|}}
\frac{\widetilde{\Theta}_{s}}{(\boldsymbol{k_s}!)(\boldsymbol{l_s}!)^{|\boldsymbol{k_s}|}} \\
=&\, \sum_{j=1}^{s}
\left(\prod_{i=0}^{j-1}
\frac{({\bf g}_{\boldsymbol{l_i}}(r_2))^{\boldsymbol{k_i}}}{(\boldsymbol{k_i}!)(\boldsymbol{l_i}!)^{|\boldsymbol{k_i}|}}\right)
\left(\frac{\widetilde{\Theta}_{j}}{(\boldsymbol{k_j}!)(\boldsymbol{l_j}!)^{|\boldsymbol{k_j}|}}\right)
\left(\prod_{i=j+1}^{s+1}
\frac{({\bf g}_{\boldsymbol{l_i}}(r_1))^{\boldsymbol{k_i}}}{(\boldsymbol{k_i}!)(\boldsymbol{l_i}!)^{|\boldsymbol{k_i}|}} \right),
\end{split}
\end{equation}
\renewcommand\theequation{C.3}
\begin{equation}
\begin{split}
\label{prop-3-3}
\widetilde{\Theta}_{j}:=&
({\bf g}_{\boldsymbol{l_j}}(r_1))^{\boldsymbol{k_j}}-({\bf g}_{\boldsymbol{l_j}}(r_2))^{\boldsymbol{k_j}} \\
=&\,
\widetilde{\Theta}_{1,j}\prod_{m=2}^{2n}(g_{\boldsymbol{l_j}}^{(m)}(r_1))^{k_{j,m}}
+(g_{\boldsymbol{l_j}}^{(1)}(r_2))^{k_{j,1}}\widetilde{\Theta}_{2,j}\prod_{m=3}^{2n}(g_{\boldsymbol{l_j}}^{(m)}(r_1))^{k_{j,m}}\\
&\, +\cdot\cdot\cdot+
\prod_{m=1}^{2n-1}(g_{\boldsymbol{l_j}}^{(m)}(r_2))^{k_{j,m}}\widetilde{\Theta}_{2n,j}
\end{split}
\end{equation}
\begin{equation*}
\begin{split}
=&\,
\sum_{i=1}^{2n}
\left(\prod_{m=0}^{i-1} (g_{\boldsymbol{l_j}}^{(m)}(r_2))^{k_{j,m}}\right)
\widetilde{\Theta}_{i,j}
\left(\prod_{m=i+1}^{2n+1}(g_{\boldsymbol{l_j}}^{(m)}(r_1))^{k_{j,m}}\right),
\end{split}
\end{equation*}
\renewcommand\theequation{C.4}
\begin{equation} \label{prop-3-4}
\begin{split}
\widetilde{\Theta}_{i,j}:
=&\,
(g^{(i)}_{\boldsymbol{l_j}}(r_1))^{k_{j,i}}-(g^{(i)}_{\boldsymbol{l_j}}(r_2))^{k_{j,i}}\\
=&\,(g^{(i)}_{\boldsymbol{l_j}}(r_1)-g^{(i)}_{\boldsymbol{l_j}}(r_2))
\sum_{m=0}^{k_{j,i}-1}(g^{(i)}_{\boldsymbol{l_j}}(r_1))^{m}(g^{(i)}_{\boldsymbol{l_j}}(r_2))^{k_{j,i}-m-1}.
\end{split}
\end{equation}
Substituting (\ref{prop-3-2})-(\ref{prop-3-4}) into (\ref{prop-3-1}),
it is easy to verify that (\ref{k-lip}) holds.



\bibliographystyle{plain}


\end{document}